\documentclass[11pt,reqno,a4paper]{amsart}
\usepackage{amssymb,amsmath,tikz}
\usepackage{mathptmx}
\usepackage[colorlinks=true]{hyperref} 
\usetikzlibrary{matrix,arrows,cd}

\newcommand{\bendL}[1]{\arrow[l,"{#1}" above,shift right, end
anchor=north east,bend right, start anchor=north west]}
\newcommand{\bendR}[1]{\arrow[l,"{#1}" below,shift left,end
anchor=south east,bend left, start anchor=south west]}

\newenvironment{diagram*}{\begin{equation*}\begin{tikzcd}}{\end{tikzcd}\end{equation*}\break}
\tikzcdset{arrows={line width=rule_thickness}}

\usepackage[matrix,arrow,curve,frame]{xy}    

\xymatrixcolsep{1.9pc}                          
\xymatrixrowsep{1.9pc}
\newdir{ >}{{}*!/-5pt/\dir{>}}                  

 \calclayout

\makeatletter

\renewcommand{\subsection}{\@startsection{subsection}{1}{0pt}{-3.25ex plus -1ex minus-.2ex}{1.5ex plus.2ex}{\normalfont\it}}
\renewcommand{\section}{\@startsection{section}{1}{\parindent}{3.5ex plus 1ex minus .2ex}{2.3ex plus.2ex}{\sc}}

\renewcommand{\phi}{\varphi}

\renewcommand{\geq}{\geqslant}
\renewcommand{\epsilon}{\varepsilon}

\renewcommand{\kappa}{\varkappa}

\DeclareMathOperator{\Sp}{\sf Sp} 
\DeclareMathOperator{\coact}{coact} \DeclareMathOperator{\act}{act}
\DeclareMathOperator{\fp}{fp} \DeclareMathOperator{\fg}{fg}
\DeclareMathOperator{\Ext}{Ext} \DeclareMathOperator{\Ch}{Ch}
\DeclareMathOperator{\spec}{Spec}

 \DeclareMathOperator{\Qcoh}{Qcoh}

\DeclareMathOperator{\Hom}{Hom} 
 \DeclareMathOperator{\id}{id}

 \DeclareMathOperator{\Mor}{Mor}

 \DeclareMathOperator{\Ab}{Ab}
 
\DeclareMathOperator{\coh}{coh} \DeclareMathOperator{\kr}{Ker}
 \DeclareMathOperator{\im}{Im}
\DeclareMathOperator{\coker}{Coker} \DeclareMathOperator{\Zg}{\mathsf Zg}
 
\DeclareMathOperator{\lm}{lim} 
 \DeclareMathOperator{\Mod}{Mod}
\DeclareMathOperator{\modd}{mod} \DeclareMathOperator{\Ob}{Ob}
\DeclareMathOperator{\supp}{Supp}

\newcommand{\lp}{\varinjlim}

\newcommand{\lo}{\lm}

\newcommand{\lra}[1]{\bl{#1}\longrightarrow\relax}
\newcommand{\bl}[1]{\buildrel #1\over}
\newcommand{\cc}{\mathcal}
\newcommand{\bb}{\mathbb}

\newcommand{\op}{{\textrm{\rm op}}}

\newcommand{\wt}{\widetilde}

\newcommand{\ac}{{}_{\cc A}\cc C}

\newcommand{\ifff}{if and only if }

\theoremstyle{plain}
\newtheorem{lem}{Lemma}[section]
\newtheorem{cor}[lem]{Corollary}
\newtheorem{prop}[lem]{Proposition}
\newtheorem{thm}[lem]{Theorem}

\theoremstyle{definition}
\newtheorem{ex}[lem]{Example}
\newtheorem{exs}[lem]{Examples}
\newtheorem{rem}[lem]{Remark}
\newtheorem{defs}[lem]{Definition}

\makeatother

\begin{document}

\footskip30pt


\title{The Ziegler spectrum for enriched ringoids and schemes}

\author{Grigory Garkusha}
\address{Department of Mathematics, Swansea University, Fabian Way, Swansea SA1 8EN, UK}
\email{g.garkusha@swansea.ac.uk}




\begin{abstract}
The Ziegler spectrum for categories enriched in closed symmetric monoidal Gro\-then\-dieck categories
is defined and studied in this paper. It recovers the classical Ziegler spectrum of a ring. As an application, the Ziegler spectrum
as well as the category of generalised quasi-coherent sheaves 
of a reasonable scheme is introduced and studied. It is shown that there is a closed embedding of 
the injective spectrum of a coherent scheme endowed with the tensor fl-topology 
(respectively of a noetherian scheme endowed with the dual Zariski topology)
into its Ziegler spectrum. It is also shown that quasi-coherent sheaves and generalised quasi-coherent 
sheaves are related to each other by a recollement.
\end{abstract}
\maketitle

\thispagestyle{empty} \pagestyle{plain}

\newdir{ >}{{}*!/-6pt/@{>}} 


\section{Introduction}

 The Ziegler spectrum of a ring $R$, defined by Ziegler in~\cite{Ziegler},
 associates a topological space ${}_R\Zg$ to $R$ whose points are the isomorphism classes 
 of indecomposable pure-injective (left) $R$-modules. A basis of quasi-compact open subsets for the Ziegler 
 topology is given by sets
    $$(\phi/\psi)=\{Q\in{}_R\Zg\mid \phi(Q) > \psi(Q)\}$$
 as $\phi/\psi$ ranges over pp-pairs (these are pairs of pp-formulas in the same free variables such that $\psi\to\phi$).
Ziegler proved in~\cite{Ziegler} that the closed sets correspond to definable subcategories of $R$-modules
(a definable subcategory $\cc D$ is sent to the closed set $\cc D\cap{}_R\Zg$).

Herzog~\cite{Herzog} and Krause~\cite{Krause} defined the Ziegler spectrum for locally coherent Grothendieck
categories in the 90s. In this language, ${}_R\Zg$ is recovered from the Ziegler spectrum of the category of generalised
$R$-modules ${}_R\cc C=(\modd R,\Ab)$ (it is a locally coherent Grothendieck category and consists of additive functors
from finitely presented right $R$-modules $\modd R$ to Abelian groups Ab). We also refer the reader to books by 
Prest~\cite{Prest,Prest2}.

Grothendieck categories of enriched functors were introduced in~\cite{AG}. Homological algebra associated to these
categories was developed in~\cite{GJ1,GJ2}. Applications of~\cite{AG,GJ1,GJ2} have recently been given
by Bonart~\cite{Bon1,Bon2} in motivic homotopy theory for reconstructing triangulated categories of Voevodsky's big
motives as well as motivic connected/very effective spectra with rational coefficients out of associated 
Grothendieck categories of enriched functors. 

Let $\cc V$ be a closed symmetric monoidal Grothendieck category.
A Grothendieck category of enriched functors is the category $[\cc A,\cc V]$ of enriched functors from
a (skeletally) small $\cc V$-category $\cc A$ to $\cc V$. In this paper we refer to $\cc A$ as an enriched ringoid similarly to 
the terminology for ringoids, which are also called pre-additive categories in the literature (i.e. categories enriched
over Abelian groups Ab). A typical example of an enriched ringoid is a DG-category, i.e. a category
enriched over the Grothendieck category
of chain complexes $\cc V=\Ch(\Ab)$. In this paper we also assume $\cc V$ to be locally finitely presented having a family
of dualizable generators $\cc G=\{g_i\}_{i\in I}$. In this case we refer to $[\cc A,\cc V]$ as the category of left $\cc A$-modules
and denote it by $\cc A\Mod$. In Section~\ref{moduli} we study basic properties of $\cc A$-modules and, more importantly, enriched
version for Serre's localization theory.

We introduce and study in Section~\ref{genmod} the enriched counterpart for the category of generalised modules.
It is the $\cc V$-category $\ac:=[\modd\cc A,\cc V]$ of enriched functors from the $\cc V$-category of finitely presented
right $\cc A$-modules $\modd\cc A$ to $\cc V$. The $\cc V$-categories $\ac$ and $\cc A\Mod$ are related to each other 
by a recollement (see Theorem~\ref{recoll} for details). Also, the $\cc V$-fully faithful functor $M\in\cc A\Mod\to-\otimes_{\cc A}M\in\ac$
identifies $\cc A$-modules with right exact functors (see Theorem~\ref{cohinjobj}). As a consequence, one shows in Section~\ref{sectionpureinj}
that this functor identifies pure-injective $\cc A$-modules introduced in Section~\ref{sectionpureinj} with injective objects of $\ac$.

Although $\ac$ is locally coherent by Theorem~\ref{vazhno}, and therefore the Ziegler spectrum 
$\Zg\ac$ in the sense of~\cite{Herzog, Krause} 
applies to $\ac$, this is actually not what we are going to investigate in this paper 
as $\Zg\ac$ does not capture the enriched
category information of $\ac$. In Section~\ref{sectionzg} we define the Ziegler spectrum ${}_{\cc A}\Zg$ of $\cc A$
that captures both the enriched category information of $\ac$ and the machinery of~\cite{Herzog, Krause}.
The points of $\Zg\ac$ and ${}_{\cc A}\Zg$ are the same but the topology on ${}_{\cc A}\Zg$ is coarser than the usual 
topology on $\Zg\ac$. 
It is worth mentioning that 
similar topologies on injective spectra of quasi-coherent sheaves which are
defined in terms of enriched functors occur in~\cite{GG3,GaPr2}. 
They play a key role for theorems reconstructing
schemes out of $\Qcoh(X)$ (see~\cite{GG3} for details).
In the case when $\cc V=\Ab$ and $\cc A=R$ is a ring,
${}_{\cc A}\Zg$ coincides with the classical Ziegler spectrum ${}_R\Zg$ of $R$.

The Ziegler topology for locally coherent Grothendieck categories defined in~\cite{Herzog,Krause} was extended
to locally finitely presented Grothendieck categories in~\cite[Theorem~11]{GG3}. In Section~\ref{sectioninj}
we define the enriched version of this topology on the injective spectrum $\Sp\cc A$ of the enriched 
ringoid $\cc A$. If $\cc V=\Qcoh(X)$ and $\cc A=\{\cc O_X\}$,
where $X$ is a reasonable scheme, $\Sp\cc A$ recovers the topological space $\Sp_{\textrm{fl},\otimes}(X)$
in the sense of~\cite[Theorem~19]{GG3}, which is equipped with the tensor fl-topology.

We finish the paper by introducing and studying the Ziegler spectrum $\Zg_X$ of a ``nice" scheme $X$. It is realised
in the category $\cc C_X$ of generalised quasi-coherent sheaves defined in Section~\ref{sectionscheme}. In this case
purity for $\cc A$-modules is the same with geometric purity studied in~\cite{EEO}. If $X$ is a nice coherent scheme,
then we show in Theorem~\ref{ugu} that there is a closed embedding $\Sp(X)\hookrightarrow\Zg_X$ 
of the injective spectrum of $X$ into its Ziegler spectrum. If, moreover, $X$ is noetherian then
there is a natural closed embedding 
of topological spaces $X^*\hookrightarrow\Zg_X$, where $X^*$ is equipped with the dual Zariski topology
(see Corollary~\ref{noeth}).

\begin{section}{Enriched Category Theory}\label{putrya}

In this section we collect basic facts about enriched categories we
shall need later. We refer the reader to~\cite{Bor,Kelly} for details.
Throughout this paper $(\mathcal{V},\otimes,\underline{\Hom},e)$ is
a closed symmetric monoidal category with monoidal product
$\otimes$, internal Hom-object $\underline{\Hom}$ and monoidal unit
$e$. We sometimes write $[a,b]$ to denote $\underline{\Hom}(a,b)$,
where $a,b\in\Ob\cc V$. We have structure isomorphisms 
   $$a_{abc}:(a\otimes b)\otimes c \to a\otimes (b\otimes c),\quad l_a:e\otimes a\to a,\quad r_a:a\otimes e\to a$$
in $\cc V$ with $a,b,c\in\Ob\cc V$.

\begin{defs}{\rm
A {\it $\mathcal{V}$-category}  $\mathcal{C}$, or {\it a category
enriched over $\mathcal{V}$}, consists of the following data:
\begin{enumerate}
\item a class $\Ob\mathcal{(C)}$ of objects;
\item for every pair $a,b \in$ $\Ob\mathcal{(C)}$ of objects, an object  $\mathcal{V}_{\cc C}(a,b)$ of
$\mathcal{V}$;
\item for every triple $a,b,c \in$ $\Ob\mathcal{(C)}$ of objects, a composition morphism in
  $\mathcal{V}$,
 $$c_{abc}:\mathcal{V}_{\cc C}(a,b) \otimes  \mathcal{V}_{\cc C}(b,c) \to  \mathcal{V}_{\cc C}(a,c);$$
\item for every object $a \in \mathcal{C}$, a unit morphism $u_a:e\to\mathcal{V}_{\cc C}(a,a)$ in
$\mathcal{V}$.

These data must satisfy the following conditions:
\begin{enumerate}
\item[$\diamond$] given objects $a,b,c,d\in\cc C$, diagram~\eqref{B6.9} below is commutative (associativity
axiom);
\item[$\diamond$] given objects $a,b\in\cc C$, diagram~\eqref{B6.10} below is commutative (unit axiom).
\end{enumerate}
\footnotesize
\end{enumerate}

\begin{equation} \label{B6.9}
\xymatrix{(\mathcal{V}_{\cc C}(a,b) \otimes \mathcal{V}_{\cc
C}(b,c)) \otimes \mathcal{V}_{\cc C}(c,d) \ar[rrr]^{c_{abc}\otimes
1} \ar[d]_{a_{\mathcal{V}_{\cc C}(a,b)\mathcal{V}_{\cc
C}(b,c)\mathcal{V}_{\cc C}(c,d)}} && &
 \mathcal{V}_{\cc C}(a,c) \otimes \mathcal{V}_{\cc C}(c,d) \ar[dd]^{c_{acd}}&   \\
\mathcal{V}_{\cc C}(a,b) \otimes (\mathcal{V}_{\cc C}(b,c) \otimes \mathcal{V}_{\cc C}(c,d)) \ar[d]_{1 \otimes{c_{bcd}}}  \\
\mathcal{V}_{\cc C}(a,b) \otimes \mathcal{V}_{\cc C}(b,d)
\ar[rrr]_{c_{abd}} & &
                         &  \mathcal{V}_{\cc C}(a,d) &    }
\end{equation}

\begin{equation} \label{B6.10}
\xymatrix{
         e \otimes \mathcal{V}_{\cc C}(a,b)  \ar[r]^{l_{\mathcal{V}_{\cc C}(a,b)}} \ar[dd]_{u_a \otimes 1}
 & \mathcal{V}_{\cc C}(a,b) \ar[dd]^{1_{\mathcal{V}_{\cc C}(a,b)}}
 & \mathcal{V}_{\cc C}(a,b) \otimes e\ar[l]_{r_{\mathcal{V}_{\cc C}(a,b)}} \ar[dd]^{1 \otimes u_{b}} \\ & & \\
         \mathcal{V}_{\cc C}(a,a) \otimes \mathcal{V}_{\cc C}(a,b) \ar[r]^(.6){c_{aab}}     & \mathcal{V}_{\cc C}(a,b) &
 \mathcal{V}_{\cc C}(a,b) \otimes \mathcal{V}_{\cc C}(b,b) \ar[l]_(.6){c_{abb}} }
\end{equation}
\normalsize When $\Ob\cc C$ is a set, the $\cc V$-category $\cc C$
is called a {\it small $\cc V$-category}. 

If there is no likelihood of confusion, we will often write $[a,b]$ for the $\cc V$-object $\mathcal{V}_{\cc C}(a,b)$.
}\end{defs}

\begin{defs} \label{functor}{\rm
Given $\mathcal{V}$-categories $\mathcal{A},\mathcal{B}$, a {\it
$\mathcal{V}$-functor\/} or an {\it enriched functor\/}
$F:\mathcal{A} \to \mathcal{B}$ consists in giving:
\begin{enumerate}
\item for every object $a \in \mathcal{A}$, an object $F(a) \in
\mathcal{B}$;
\item for every pair $a,b \in \mathcal{A}$ of objects, a morphism in  $\mathcal{V},$
$$F_{ab}:\mathcal{V}_{\cc A}(a,b) \to  \mathcal{V}_{\cc B}(F(a),F(b))$$
in such a way that the following axioms hold:
\begin{enumerate}
\item[$\diamond$] for all objects $a,a',a''\in\cc A$, diagram~\eqref{B6.11} below commutes (composition
axiom);
\item[$\diamond$] for every object $a\in\cc A$, diagram~\eqref{fun} below commutes (unit
axiom).
\end{enumerate}
\begin{equation}    \label{B6.11}
\xymatrix{ \cc V_\cc A(a,a')\otimes\cc V_\cc
A(a',a'')\ar[rr]^(.6){c_{aa'a''}}\ar[d]_{F_{aa'}\otimes F_{a'a''}}
&&\cc V_\cc A(a,a'')\ar[d]^{F_{aa''}}\\
\cc V_\cc B(Fa,Fa')\otimes\cc V_\cc
B(Fa',Fa'')\ar[rr]_(.6){c_{Fa,Fa',Fa''}}&&\cc V_\cc B(Fa,Fa'') }
\end{equation}
\begin{equation}    \label{fun}
\xymatrix{
e\ar[r]^(.4){u_a}\ar[dr]_{u_{Fa}}&\cc V_\cc A(a,a)\ar[d]^{F_{aa}}\\
&\cc V_\cc B(Fa,Fa) }
\end{equation}
\end{enumerate}
}\end{defs}

\begin{defs}\label{trans}{\rm
Let $\cc A, \cc B$ be two $\cc V$-categories and $F,G:\cc A\to \cc
B$ two $\cc V$-functors. A {\it $\cc V$-natural transformation\/}
$\alpha:F \Rightarrow G$ consists in giving, for every object
$a\in\cc A$, a morphism
    $$\alpha_a:e \to \cc V_\cc B(F(a),G(a))$$
in $\cc V$ such that diagram~\eqref{B6.13} below commutes, for all
objects $a,a'\in\cc A$.
\begin{equation}    \label{B6.13}
\xymatrix{
&\cc V_{\cc A}(a,a')\ar[ddl]_{l^{-1}_{\cc V_{\cc A}(a,a')}}\ar[ddr]^{r^{-1}_{\cc V_{\cc A}(a,a')}}&\\
&&\\
e\otimes\cc V_{\cc A}(a,a')\ar[d]_{\alpha_a \otimes G_{aa'}}&
&\cc V_{\cc A}(a,a')\otimes e\ar[d]^{F_{aa'}\otimes \alpha_{a'}}\\
\cc V_{\cc B}(Fa,Ga)\otimes\cc V_{\cc
B}(Ga,Ga')\ar[ddr]_{c_{FaGaGa'}\text{ }}
&&\cc V_{\cc B}(Fa,Fa')\otimes\cc V_{\cc B}(Fa',Ga')\ar[ddl]^(.45){\text{       }\text{  }  c_{FaFa'Ga'}}\\
&&\\
&\cc V_{\cc B}(Fa,Ga')& }
\end{equation}

}\end{defs}

By $\mathbf{Set}$ we shall mean the closed symmetric monoidal
category of sets. Categories in the usual sense are
$\mathbf{Set}$-categories (categories enriched over $\mathbf{Set}$).
If $\cc A$ is a category, let $\mathbf{Set}_{\cc A} (a,b)$ denote
the set of maps in $\cc A$ from $a$ to $b$. The closed symmetric
monoidal category $\cc V$ is a $\cc V$-category due to its internal
$\Hom$-objects. Any $\cc V$-category $\cc C$ defines a
$\mathbf{Set}$-category ${\cc{C}_0}$, also called the {\it
underlying category}. Its class of objects is $\Ob \cc C$, the
morphism sets are $\mathbf{Set}_{{\cc{C}_0}}{(a,b)}=\mathbf{Set}_{\cc
V}(e, \cc{V_C}{(a,b))}$ (see~\cite[p.~316]{Bor}).

\begin{prop}\label{rrr}
Let $\cc V$ be a symmetric monoidal closed category. If $\cc A$ is a
$\cc V$-category and $F,G:\cc A\rightrightarrows \cc V$ are $\cc
V$-functors, giving a $\cc V$-natural transformation
$\alpha:F\Rightarrow G$ is equivalent to giving a family of morphisms
$\alpha:F(a) \to G(a)$ in $\cc V$, for $a \in \cc A$, in such a way
that the following diagram commutes for all $a,a' \in \cc A$

$$ \xymatrix{\cc V_\cc A (a,a') \ar[r]^{F_{aa'}} \ar[d]_{G_{aa'}}
& [F(a),F(a')] \ar[d]^{[1,\alpha_{a'}]} \\
             [G(a),G(a')] \ar[r]_{[\alpha_a,1]} & [F(a),G(a')]}$$
\end{prop}

\begin{proof}
See~\cite[Proposition~6.2.8]{Bor}.
\end{proof}

\begin{cor}\label{good}
Let $\cc V$ be a symmetric monoidal closed category. If $ \cc A$ is
a $\cc V$-category and $F,G:\cc A \rightrightarrows \cc V$ are $\cc
V$-functors, giving a $\cc V$-natural transformation
$\alpha:F\Rightarrow G$ is equivalent to giving a family of morphisms
$\alpha:F(a) \to G(a)$ in $\cc V$, for $a \in \cc A$, in such a way
that the following diagram commutes for all $a,a' \in \cc A$
   $$ \xymatrix{\cc V_\cc A (a,a') \otimes F(a) \ar[r]^(.65){\eta_F}  \ar[d]_{1 \otimes \alpha_a} &
       F(a') \ar[d]^{\alpha_{a'}} \\ \cc V_\cc A(a,a') \otimes G(a) \ar[r]_(.65){\eta_G} & G(a'),}$$
where $\eta_F,\eta_G$ are the morphisms corresponding to the
structure morphisms $F_{aa'}$ and $G_{aa'}$ respectively.
\end{cor}

\begin{cor}\label{wow}
Let $\cc V$ be a symmetric monoidal closed category. If $\cc A$ is a
small $\cc V$-category and $F,G:\cc A\rightrightarrows \cc V$ are
$\cc V$-functors, suppose $\alpha:F\Rightarrow G$ is a $\cc
V$-natural transformation such that each $\alpha_a:F(a) \to G(a)$,
$a\in\Ob\cc A$, is an isomorphism in $\cc V$. Then $\alpha$ is an
isomorphism in $[\cc A,\cc V]$.
\end{cor}

\begin{proof}
See~\cite[Corollary~2.6]{AG}.
\end{proof}

Let $\cc C, \cc D$ be two $\cc V$-categories. The {\it monoidal
product\/} $\cc C\otimes \cc D$ is the $\cc V$-category, where
   $$\Ob(\cc C\otimes \cc D):=\Ob\cc C\times\Ob\cc D$$
and
   $$\cc V_{\cc C\otimes \cc D}((a, x),(b, y)):=\cc V_{\cc C}(a, b)\otimes\cc V_{\cc D}(x, y),\quad a,b\in\cc C,x,y\in\cc D.$$

\begin{defs}\label{enri}{\rm
A $\cc V$-category $\cc C$ is a {\it right $\cc V$-module} if there
is a $\cc V$-functor $\act:\cc C\otimes \cc V \to \cc C$, denoted
$(c,A) \mapsto c\oslash A$ and a $\cc V$-natural unit isomorphism
$r_c:\act(c,e)\to c$ subject to the following conditions:
\begin{enumerate}
\item there are coherent natural associativity isomorphisms $c\oslash (A \otimes B) \to (c \oslash A) \otimes
B$;
\item the isomorphisms $c\oslash (e \otimes A)\rightrightarrows c \oslash A$ coincide.

\end{enumerate}
A right $\cc V$-module is {\it closed\/} if there is a $\cc
V$-functor
$$\coact:\cc V^{\op} \otimes \cc C \to \cc C$$
such that for all $A \in \Ob \cc V$, and $c \in \Ob\cc C$, the $\cc
V$-functor $\act(-,A):\cc C \to \cc C$ is left
 $\cc V$-adjoint to $\coact(A,-)$ and $\act(c,-):\cc V \to \cc C$
is left $\cc V$-adjoint to $\cc V_{\cc C}(c,-)$. }
\end{defs}

If $\cc C$ is a small $\cc V$-category, $\cc V$-functors from $\cc
C$ to $\cc V$ and their $\cc V$-natural transformations form the
category $[\cc C,\cc V]$ of $\cc V$-functors from $\cc C$ to $\cc
V$. If $\cc V$ is complete, then $[\cc C,\cc V]$ is also a $\cc
V$-category. We also denote this $\cc V$-category by
$\cc{F(C)}$, or $\cc F$ if no confusion can arise. The morphism $\cc
V$-object $\cc{V_F}(X,Y)$ is the end
   \begin{equation}\label{theend}
    \int_{\Ob \cc C} \cc V (X(c),Y(c)).
   \end{equation}
Note that the underlying category $\cc{F}_0$ of the $\cc V$-category
$\cc F$ is $[\cc C,\cc V]$.

Given $c \in \Ob \cc C$, $X\mapsto X(c)$ defines the $\cc V$-functor
$\text{Ev}_c:\cc F\to\cc V$ called {\it evaluation at c}. The
assignment $c\mapsto\cc V_\cc C (c,-)$ from $\cc C$ to $\cc F$ is
again a $\cc V$-functor $\cc C^{\op}\to \cc F$, called the {\it $\cc
V$-Yoneda embedding}. $\cc V_\cc C (c,-)$ is a representable
functor, represented by $c$.

\begin{lem}[The Enriched Yoneda Lemma]\label{enryon}
Let $\cc V$ be a complete closed symmetric monoidal category and
$\cc C$ a small $\cc V$-category. For every $\cc V$-functor $X:\cc C
\to \cc V$ and every $c\in \Ob \cc C$, there is a $\cc V$-natural
isomorphism $X(c) \cong \cc V_\cc F (\cc V_\cc C (c,-),X)$.
\end{lem}

\begin{lem}\label{bicomplete}
If $\cc V$ is a bicomplete closed symmetric monoidal category and
$\cc C$ is a small $\cc V$-category, then $[\cc C,\cc V]$ is
bicomplete. (Co)limits are formed pointwise. Moreover,
$\cc F$ is a closed $\cc V$-module.
\end{lem}

\begin{proof}
See~\cite[Proposition~6.6.17]{Bor}.
\end{proof}

\begin{thm}\label{polezno}
Assume $\cc V$ is bicomplete, and let $\cc C$ be a small $\cc
V$-category. Then any $\cc V$-functor $X:\cc C\to\cc V$ is $\cc
V$-naturally isomorphic to the coend
    $$X\cong\int^{\Ob\cc C}\cc V_\cc C(c,-)\oslash X(c).$$
\end{thm}

\begin{proof}
See~\cite[Theorem~6.6.18]{Bor}.
\end{proof}

A {\it monoidal $\cc V$-category\/} is a $\cc V$-category $\cc C$
together with a $\cc V$-functor $\diamond:\cc C\otimes\cc C\to\cc
C$, a unit $u\in\Ob\cc C$, a $\cc V$-natural associativity
isomorphism and two $\cc V$-natural unit isomorphisms. Symmetric
monoidal and closed symmetric monoidal $\cc V$-categories are
defined similarly.

Suppose $(\cc C,\diamond,u)$ is a small symmetric monoidal $\cc
V$-category, where $\cc V$ is bicomplete. In~\cite{Day}, a closed
symmetric monoidal product was constructed on the category $[\cc
C,\cc V]$ of $\cc V$-functors from $\cc C$ to $\cc V$. For
$X,Y\in\Ob [\cc C,\cc V]$, the monoidal product  $X\odot Y\in\Ob[\cc
C,\cc V]$ is the coend
\begin{equation}\label{mon}
 X\odot Y:=\int^{\Ob (\cc C\otimes\cc C)}\cc V_\cc C (c\diamond d,-)\otimes (X(c)\otimes Y(d)):\cc C\to\cc V.
\end{equation}

The following theorem is due to Day~\cite{Day} and plays an
important role in our analysis.

\begin{thm}[Day~\cite{Day}]\label{daythm}
Let $(\cc V,\otimes ,e)$ be a bicomplete closed symmetric monoidal
category and $(\cc C,\diamond,u)$ a small symmetric monoidal $\cc
V$-category. Then $([\cc C,\cc V],\odot,\cc V_\cc C (u,-))$ is a
closed symmetric monoidal category. The  internal $\Hom$-functor in
$[\cc C,\cc V]$ is given by the end
\begin{equation}\label{inthom}
  \cc F(X,Y)(c)=\cc V_\cc F(X,Y(c\diamond-))=\int_{d\in\Ob\cc C}\cc V(X(d),Y(c\diamond d)).
\end{equation}
\end{thm}

The next lemma computes the tensor product of representable
$\cc V$-functors.

\begin{lem}\label{winner}
The tensor product of representable functors is again representable.
Precisely, there is a natural isomorphism
   $$\cc V_\cc C(c,-)\odot\cc V_\cc C (d,-)\cong\cc V_\cc C(c\diamond d,-).$$
\end{lem}

\end{section}

\begin{section}{Grothendieck categories}

In this section we collect basic facts about Grothendieck
categories. We mostly follow Herzog~\cite{Herzog} and
Stenstr\"om~\cite{Stenstroem}.

\begin{defs}{\rm
A family $\{U_i\}_I$ of objects of an Abelian category $\cc A$ is a
{\it family of generators\/} for $\cc A$ if for each non-zero
morphism $\alpha:B \to C$ in $\cc A$ there exist a morphism $\beta:
U_i\to B$ for some $i \in I$, such that $\alpha \beta \ne 0$.

}\end{defs}

Recall that an Abelian category is {\it cocomplete} or an {\it
Ab3-category} if it has arbitrary direct sums. The cocomplete
Abelian category $\cc C$ is said to be an {\it Ab5-category} if for
any directed family $\{A_i\}_{i\in I}$ if subobjects of $A$ and for
any subobject $B$ of $A$, one has
    $$(\sum_{i\in I}A_i)\cap B=\sum_{i\in I}(A_i \cap B).$$

The condition $Ab3$ is equivalent to the existance of arbitrary
direct limits. Also $Ab5$ is equivalent  to the fact that there
exist inductive limits and the inductive limits over directed
families of indices are exact, i.e. if $I$ is a directed set and
\begin{equation*}
\xymatrix{
0\ar[r]&A_i\ar[r]&B_i\ar[r]&C_i\ar[r]&0
}
\end{equation*}
is an exact sequence for any $i\in I$, then
\begin{equation*}
\xymatrix{
0\ar[r]&\varinjlim A_i\ar[r]&\varinjlim B_i\ar[r]&\varinjlim C_i\ar[r]&0
}
\end{equation*}
is an exact sequence.

An Abelian category which satisfies the condition $Ab5$ and which
possesses a family of generators is called a {\it Grothendieck
category}.

\begin{defs} {\rm
Recall that an object $A\in\cc C$ is {\em finitely generated\/} if
whenever there are subobjects $A_i\subseteq A$ for $i\in I$
satisfying $A=\sum_{i\in I}A_i$, then there is a finite subset
$J\subset I$ such that $A=\sum_{i\in J}A_i$. The category of
finitely generated subobjects of $\cc C$ is denoted by $\fg\cc C$.
The category is {\em locally finitely generated\/} provided that
every object $X\in\cc C$ is a directed sum $X=\sum_{i\in I}X_i$ of
finitely generated subobjects $X_i$, or equivalently, $\cc C$
possesses a family of finitely generated generators.}
\end{defs}

\begin{defs} {\rm
A finitely generated object $B\in\cc C$ is {\em finitely
presented\/} provided that every epimorphism $\eta:A\to B$ with $A$
finitely generated has a finitely generated kernel $\kr\eta$. The
subcategory of finitely presented objects of $\cc C$ is denoted by
$\fp\cc C$. Note that the
subcategory $\fp\cc C$ of $\cc C$ is closed under extensions.
Moreover, if
   $$0\to A\to B\to C\to 0$$
is a short exact sequence in $\cc C$ with $B$ finitely presented,
then $C$ is finitely presented \ifff $A$ is finitely generated. The
category is {\em locally finitely presented\/} provided that it is
locally finitely generated and every object $X\in\cc C$ is a direct
limit $X=\lp_{i\in I}X_i$ of finitely presented objects $X_i$, or
equivalently, $\cc C$ possesses a family of finitely presented
generators.
}
\end{defs}

\begin{defs} {\rm
A finitely presented object $C$ of a locally finitely presented
Grothen\-dieck category $\cc C$ is {\it coherent\/} if every
finitely generated subobject $B$ of $C$ is finitely presented.
Equivalently, every epimorphism $h:C \to A$ with $A$ finitely
presented has a finitely presented kernel. Evidently, a finitely
generated subobject of a coherent object is also coherent. The
subcategory of coherent objects of $\cc C$ is denoted by $\coh \cc
C$. The category $\cc C$ is {\em locally coherent\/} provided that
it is locally finitely presented and every object $X\in\cc C$ is a
direct limit $X=\lp_{i\in I}X_i$ of coherent objects $X_i$, or
equivalently, $\cc C$ possesses a family of coherent generators.}
\end{defs}

The subcategories consisting of finitely generated, finite\-ly
presented and coherent objects are ordered by inclusion as follows:
   $$\cc C\supseteq\fg\cc C\supseteq\fp\cc C\supseteq\coh\cc C.$$

Suppose $\cc V$ is a
closed symmetric monoidal Grothendieck category.
Here are some examples of such categories.

\begin{exs}\label{primery}
(1) Given any commutative ring $R$, the triple $(\Mod R,
\otimes_R,R)$ is a closed symmetric monoidal locally finitely presented
Grothendieck category.

(2) More generally, let $X$ be a quasi-compact quasi-separated
scheme. Consider the category $\Qcoh(X)$ of quasi-coherent
$\cc O_X$-modules. By~\cite[Lemma~3.1]{Illusie} $\Qcoh(X)$ is a
Grothendieck category, where
quasi-coherent $\cc O_X$-modules of finite type form generators. 
It is locally finitely presented by~\cite[Proposition~7]{GG3}. The tensor product of $\cc
O_X$-modules preserves quasi-coherence, and induces a closed
symmetric monoidal structure on $\Qcoh(X)$.

(3) By~\cite[Proposition~3.4]{AG} the category of unbounded chain complexes $\Ch(\cc V)$ of a
Grothen\-dieck category $\cc V$ is again a Grothendieck category. If, in addition,
$\cc V$ is closed symmetric monoidal, $\Ch(\cc V)$ is with respect to the usual
monoidal product and internal Hom-object by~\cite[Theorem~3.2]{GJ1}. Moreover,
if $\cc V$ is locally finitely presented, then $\Ch(\cc V)$ is.

(4) ($\Mod kG,\otimes_k,k)$ is a closed symmetric monoidal locally finitely presented
Grothendieck category, where $k$ is a field and $G$ is a finite
group.
\end{exs}

\begin{thm}[see~\cite{AG}]\label{pqp}
Let $\mathcal{V}$ be a closed symmetric monoidal Grothendieck
category with a set of generators $\{g_i\}_I$. If $\mathcal{C}$ is a
small $\mathcal{V}$-category, then the category of enriched functors
$[\mathcal C,\mathcal V]$ is a Grothendieck $\cc V$-category with
the set of generators $\{\mathcal{V}(c,-) \oslash g_i\mid c\in\Ob\cc
C,i\in I\}$. Moreover, if $\cc C$ is a small symmetric monoidal $\cc
V$-category, $[\mathcal C,\mathcal V]$ is closed symmetric
monoidal with tensor product and internal $\Hom$-object computed
by the formulas~\eqref{mon} and~\eqref{inthom} of Day.
\end{thm}

We say that a full subcategory $\cc S$ of an Abelian category $\cc
C$ is a {\it Serre subcategory\/} if for any short exact sequence
   $$0\to X\to Y\to Z\to 0$$
in $\cc C$ an object $Y\in\cc S$ if and only if $X$, $Z\in\cc S$. A
Serre subcategory $\cc S$ of a Grothendieck category $\cc C$ is {\it
localizing\/} if it is closed under taking direct limits.
Equivalently, the inclusion functor $i:\cc S\to\cc C$ admits the
right adjoint $t=t_{\cc S}:\cc C\to\cc S$ which takes every object
$X\in\cc C$ to the maximal subobject $t(X)$ of $X$ belonging to $\cc
S$. The functor $t$ we call the {\it torsion functor}. An object $C$
of $\cc C$ is said to be {\it $\cc S$-torsionfree\/} if $t(C)=0$.
Given a localizing subcategory $\cc S$ of $\cc C$ the {\it quotient
category $\cc C/\cc S$\/} consists of $C\in\cc C$ such that
$t(C)=t^1(C)=0$, where $t^1$ stands for the first derived functor
associated with $t$. The objects from $\cc C/\cc S$ we call {\it
$\cc S$-closed objects}. Given $C\in\cc C$ there exists a canonical
exact sequence
$$0\to A'\to C\lra{\lambda_C}C_{\cc S}\to A''\to 0$$
with $A'=t(C)$, $A''\in\cc S$, and where $C_{\cc S}\in\cc C/\cc S$
is the maximal essential extension of $\wt C=C/t(C)$ such that
$C_{\cc S}/\wt C\in\cc S$. The object $C_{\cc S}$ is uniquely
defined up to a canonical isomorphism and is called the {\it $\cc
S$-envelope\/} of $C$. Moreover, the inclusion functor $\iota:\cc
C/\cc S\to\cc C$ has the left adjoint {\it localizing functor\/}
$(-)_{\cc S}:\cc C\to\cc C/\cc S$, which is also exact. It takes
each $C\in\cc C$ to $C_{\cc S}\in\cc C/\cc S$. Then,
   $$\Hom_{\cc C}(X,Y)\cong\Hom_{\cc C/\cc S}(X_{\cc S},Y)$$
for all $X\in\cc C$ and $Y\in\cc C/\cc S$.

\end{section}

\section{The category of $\cc A$-modules}\label{moduli}

{\it From now on we fix a closed symmetric monoidal Grothendieck category $\cc V$
and assume that it has a family of dualizable 
generators $\cc G=\{g_i\}_{i\in I}$ such that the monoidal unit $e$ is finitely presented}.
Under these assumptions, each $g\in\cc G$ is finitely presented as well.
In particular, $\cc V$ is locally finitely presented.

Recall from~\cite[6.5]{HO} that 
the following conditions are equivalent for an object $p\in\cc V$: 

\begin{enumerate}
\item The canonical morphism $p^{\vee}\otimes p\to[p,p]$ is an isomorphism, where $p^\vee=[p,e]$.
\item The canonical morphism $p^\vee\otimes z\to [p, z]$ is an isomorphism for all $z\in\cc V$. 
\item $[p,y]\otimes z\to[p,y\otimes z]$ is an isomorphism for all $y,z\in\cc V$.
\end{enumerate}
Such an object $p\in\cc V$ is said to be {\it dualizable}. By
\cite[Lemma~6.7]{HO} if $p,q\in\cc V$ are dualizable, then so are $p\oplus q$, $p\otimes q$, and $[p,q]$.

\begin{exs}\label{examples}
(1) If $R$ is a commutative ring, then $R$ is a dualizable generator for $\Mod R$.

(2) By Example~\ref{primery}(2) $\Qcoh(X)$ is a closed symmetric monoidal locally finitely presented
Gro\-then\-dieck category over a quasi-compact quasi-separated scheme $X$. 
By~\cite[Proposition~4.7.5]{Bran} dualizable objects of $\Qcoh(X)$ are precisely locally free of finite rank.
We say that $X$ satisfies the {\it strong resolution property\/} if they also generate $\Qcoh(X)$.
It is worth noticing that Corollary~\ref{zabavnocor} below and~\cite[Main Theorem]{SlSt} imply $X$ must be semiseparated.

(3) By Example~\ref{primery}(3) $\Ch(\cc V)$ is a closed symmetric monoidal locally finitely presented
Gro\-then\-dieck category. It is directly verified that $\Ch(\cc V)$ has dualizable generators. Precisely, they are
dual to the standard generators of $\Ch(\cc V)$.

(4) By Example~\ref{primery}(4) $\Mod kG$ is a closed symmetric monoidal locally finitely presented
Gro\-then\-dieck category, where $k$ is a field and $G$ is a finite group. It is well known that finite dimensional $kG$-modules
are dualizable (see, e.g.,~\cite[\S2.14]{Iy}).
\end{exs}

In what follows we assume that the set of dualizable generators $\cc G$ is closed under tensor products
$g\otimes g'$ and dual objects $g^\vee$. $\cc G$ becomes a symmetric monoidal $\cc V$-category
if we assume $e\in\cc G$.
The category $[\cc G^{\op},\cc V]$ is then closed symmetric 
monoidal with respect to the Day monoidal product $\odot$ of Theorem~\ref{daythm}.

The following statement is reminiscent of the fact saying that the category of modules $\Mod R$ over a commutative ring $R$
is equivalent to the category of additive functors $(\cc P^{\op},\Mod R)$, where $\cc P$ is the full subcategory of
finitely generated projective $R$-modules (note that each $P\in\cc P$ is dualizable in $\Mod R$).

\begin{prop}\label{zabavno}
Suppose $e\in\cc G$. The adjoint pair $-\otimes [-,e]:\cc V\leftrightarrows[\cc G^{\op},\cc V]:Ev_e$ is an equivalence of categories. Likewise,
the adjoint pair $[e,-]\otimes-:\cc V\leftrightarrows[\cc G,\cc V]:Ev_e$ is an equivalence of categories.
\end{prop}

\begin{proof}
If we identify $\cc V$ with $[e,\cc V]$, the functor $-\otimes [-,e]$ is nothing but the enriched left Kan extension
sending $X\in\cc V$ to $[-,X]$,
hence it is fully faithful. One also has $(X\otimes[-,e])(g)=X\otimes g^\vee\cong[g,X]$.

Suppose $Y\in[\cc G^{\op},\cc V]$, then there is a canonical isomorphism for all $g\in\cc G$:
   $$Ev_e(Y\odot (g^\vee\otimes[-,e]))\cong Ev_e((Y\odot[-,e])\otimes g^\vee)\cong Ev_e(Y)\otimes g^\vee=Y(e)\otimes g^\vee.$$
One has, 
   \begin{multline*}
    Ev_e(Y\odot (g^\vee\otimes[-,e]))\cong Ev_e(Y\odot[-,g^\vee])\cong Ev_e(\int^{g'\in\cc G}(Y(g')\otimes[-,g']\odot[-,g^\vee]))=\\
    =\int^{g'\in\cc G}Y(g')\otimes[e,g'\otimes g^\vee]\cong\int^{g'\in\cc G}Y(g')\otimes[g,g']\cong Y(g).
   \end{multline*}
We use Lemma~\ref{winner} here. We see that there is a canonical isomorphism $Y(g)\cong Y(e)\otimes g^\vee\cong[g, Y(e)]$.
It follows that the adjunction unit morphism $Y\to Ev_e(Y)\otimes[-,e]$ is an isomorphism at every $g\in\cc G$, and hence it is an 
isomorphism by Corollary~\ref{wow}.
\end{proof}

\begin{cor}\label{zabavnocor}
For any dualizable generator $g\in\cc G$ the functors $g\otimes-,[g,-]:\cc V\to\cc V$ are exact. In particular,
the generator $G:=\oplus_{g\in\cc G}g$ of $\cc V$ is flat in the sense that $G\otimes-:\cc V\to\cc V$ is an exact functor.
\end{cor}

\begin{proof}
Without loss of generality we may assume $e\in\cc G$, because $e$ is dualizable and $\cc G\cup\{e\}$ remains a family of generators of $\cc V$. 
The proof of the preceding proposition shows that the equivalence $-\otimes[-,e]:\cc V\to[\cc G^{\op},\cc V]$ is isomorphic to
$X\mapsto[-,X]$. As it preserves exact sequences, the functor $[g,-]$ is exact. Likewise, the functor $g\otimes-$ is exact.
\end{proof}

Recall that a preadditive category is a category enriched over $\Ab$. It is also called a {\it ring with several objects\/}
or a {\it ringoid\/} in the literature. 
Likewise, we refer to a (skeletally) small $\cc V$-category $\cc A$ as an {\it enriched ring with several objects\/} or as
an {\it enriched ringoid}. A typical example of an enriched ringoid 
is a DG-category, in which case $\cc V=\Ch(\Ab)$. In 
what follows the category of contravariant (respectively covariant) $\cc V$-functors $[\cc A^{\op},\cc V]$ 
(respectively $[\cc A,\cc V]$) will be denoted by $\Mod\cc A$ (respectively $\cc A\Mod$). If there is no likelihood of
confusion, we refer to objects of $\Mod\cc A$ (respectively $\cc A\Mod$) 
as {\it right $\cc A$-modules} (respectively {\it left $\cc A$-modules}).

\begin{lem}\label{fpr}
The Grothendieck $\cc V$-categories $\Mod\cc A$ and $\cc A\Mod$ are locally finitely presented with 
finitely presented generators $\{g\oslash[-,a]\mid a\in\cc A,g\in\cc G\}$ and $\{[a,-]\oslash g\mid a\in\cc A,g\in\cc G\}$
respectively.
\end{lem}

\begin{proof}
This follows from Theorem~\ref{pqp} and the fact that 
   \begin{multline*}
    \Hom_{\cc A}([a,-]\oslash g,\lp_I M_i)=\Hom_{\cc V}(g,\lp_I M_i(a))=\\
    =\lp_I\Hom_{\cc V}(g,M(a))=\lp_I\Hom_{\cc A}([a,-]\oslash g,M_i)
   \end{multline*}
for any $a\in\cc A$ and $g\in\cc G$.
\end{proof}

The proof of the preceding lemma also leads to the following useful lemma.

\begin{lem}\label{fprV}
A left or right $\cc A$-module $M$ is finitely presented if and only if the canonical morphism
$\lp_I[M,X_i]\to[M,\lp_IX_i]$ is an isomorphism in $\cc V$ for any direct limit of $\cc A$-modules.
\end{lem}

\begin{proof}
Without loss of generality we may assume $e\in\cc G$, because otherwise we add $e$ to $\cc G$.
As each generator $g\in\cc G$
is finitely presented and $\cc V$ is equivalent to $[\cc G^{\op},\cc V]$ by Proposition~\ref{zabavno} by means of the functor
$X\mapsto[-,X]$, every direct limit
in $\cc V$ is mapped to a direct limit in $[\cc G^{\op},\cc V]$. Therefore $\lp_I[g,X_i]\to[g,\lp_IX_i]$ is an isomorphism in $\cc V$ for any
$g\in\cc G$.

Let $M\in\Mod\cc A$ be finitely presented. There is an exact sequence in $\Mod\cc A$
    $$\oplus_{s=1}^mg_s\oslash[-,a_s]\to\oplus_{t=1}^ng_t\oslash[-,a_t]\to M\to 0,\quad g_s,g_t\in\cc G,\quad a_s,a_t\in\cc A.$$
It induces an exact sequence
   $$0\to[M,\lp_IX_i]\to\lp_I\oplus_{t=1}^n[g_t,X_i(a_t)]\to\lp_I\oplus_{s=1}^m[g_s,X_i(a_s)].$$
It follows that the canonical morphism
$\lp_I[M,X_i]\to[M,\lp_IX_i]$ is an isomorphism in $\cc V$. The converse is straightforward if we apply
$\Hom_{\cc V}(e,-)$ to the isomorphism $\lp_I[M,X_i]\cong[M,\lp_IX_i]$ and use the assumption that $e\in\fp(\cc V)$.
\end{proof}

As $\Mod\cc A$ and $\cc A\Mod$ are closed $\cc V$-modules by Lemma~\ref{bicomplete}, the functors
$[X,-]$ and $[-,X]$ are left exact for any $\cc A$-module $X$. It will be useful to have the following

\begin{lem}\label{tochen}
If $E$ is an injective left or right $\cc A$-module, then the functor $[-,E]$ is exact.
\end{lem}

\begin{proof}
A sequence of $\cc A$-modules $0\to A\to B\to C\to 0$ is short exact if and only if the sequence
$0\to A(a)\to B(a)\to C(a)\to 0$ is short exact in $\cc V$ for any $a\in\cc A$. It follows from Corollary~\ref{zabavnocor}
that $0\to G\oslash A\to G\oslash B\to G\oslash C\to 0$ is short exact, where $G=\oplus_{g\in\cc G}g$ is a generator of $\cc V$.
As the functor $\Hom_{\cc V}(G,-):\cc V\to\Ab$ is faithful, it reflects epimorphisms by~\cite[Proposition~1.2.12]{KS}. 

So $[B,E]\to[A,E]$ is an epimorphism in $\cc V$ if $\Hom_{\cc V}(G,[B,E])\to\Hom_{\cc V}(G,[A,E])$
is an epimorphism. The latter arrow is an epimorphism in Ab as it is isomorphic to the epimorphism 
$\Hom_{\cc V}(G\oslash B,E)\to\Hom_{\cc V}(G\oslash A,E)$ (we use here our assumption that $E$ is injective).
\end{proof}

\begin{lem}\label{gE}
If $E$ is an injective left or right $\cc A$-module and $g\in\cc G$, then $[g,E]$ and $g\oslash E$ are injective $\cc A$-modules.
\end{lem}

\begin{proof}
As $[g,E]\cong g^{\vee}\oslash E$, it is enough to show that $[g,E]$ is injective. But this immediately follows from
isomorphisms of the form $\Hom_{\cc A}(A,[g,E])\cong\Hom_{\cc A}(g\oslash A,E)$, where $A$ is an $\cc A$-module, 
and Corollary~\ref{zabavnocor}.
\end{proof}

A localising subcategory $\cc P\subseteq\Mod\cc A$ is said to be {\it enriched\/} if $M\oslash V\in\cc P$
for any $V\in\cc V$ and $M\in\cc P$. The quotient category $\Mod\cc A/\cc P$ will also be regarded as a full 
$\cc V$-subcategory of the $\cc V$-category $\Mod\cc A$. Denote by $\underline{\Ext}^1(P,-)$ the first derived functor
associated with the left exact functor $[P,-]:\Mod\cc A\to\cc V$.

\begin{thm}\label{svoistva}
The following statements are true for an enriched localising subcategory $\cc P$:
\begin{enumerate}
\item for any $\cc P$-closed $\cc A$-module $M\in\Mod\cc A/\cc P$ and any $V\in\cc V$,
the $\cc A$-module $[V,M]$ is $\cc P$-closed;
\item the localisation functor $(-)_{\cc P}:\Mod\cc A\to\Mod\cc A/\cc P$ is a $\cc V$-functor;
\item the quotient category $\Mod\cc A/\cc P$ is a closed $\cc V$-module, where tensor (respectively cotensor) objects are defined
as $(M\oslash V)_{\cc P}$ (respectively $[V,M]$) for any $\cc P$-closed $\cc A$-module $M$ and $V\in\cc V$;
\item an $\cc A$-module $M$ is $\cc P$-closed if and only if $[P,M]=\underline{\Ext}^1(P,M)=0$ for every $P \in {\mathcal P}$.
\end{enumerate}
\end{thm}

\begin{proof}
(1). Recall that an injective $\cc A$-module $E$ is $\cc P$-closed if and only if it is $\cc P$-torsionfree. By Lemma~\ref{gE}
$[g,E]$ is injective for all $g\in\cc G$. As $\cc P$ is enriched, $[g,E]$ is $\cc P$-torsionfree, and hence $\cc P$-closed for any $\cc P$-closed
injective $E$. Suppose
$V\in\cc V$. There is an exact sequence 
   $$\oplus_Ig_i\to\oplus_Jg_j\to V\to 0$$
in $\cc V$. It induces an exact sequence in $\Mod\cc A$
   $$0\to[V,E]\to\prod_J[g_j,E]\to\prod_I[g_i,E].$$
As $[V,E]$ is the kernel of a map between $\cc P$-closed modules, it is $\cc P$-closed itself. Next, each $\cc P$-closed
module $M$ fits into an exact sequence
   $$0\to M\to E_1\to E_2$$
with $E_1,E_2$ injective $\cc P$-closed modules. It induces an exact sequence in $\Mod\cc A$
   $$0\to[V,M]\to[V,E_1]\to[V,E_2].$$
Since $[V,E_1],[V,E_2]$ are $\cc P$-closed, $[V,M]$ is.

(2). Suppose $A,B\in\Mod\cc A$, then a morphism of $\cc V$-objects $[A,B]\to[A_{\cc P},B_{\cc P}]$ is determined by
a morphism of representable functors $(-,[A,B])\to(-,[A_{\cc P},B_{\cc P}])$ if we use Yoneda's lemma~\cite[Theorem~1.3.3]{Bor1}.
Given $V\in\cc V$, the definition of an arrow $(V,[A,B])\to(V,[A_{\cc P},B_{\cc P}])$ is equivalent to defining an arrow
$(A,[V,B])\to(A,[V,B_{\cc P}])$ due to the fact that $[V,B_{\cc P}]$ is $\cc P$-closed by the previous statement. The latter
arrow is induced by the canonical map $\lambda_B:B\to B_{\cc P}$. This defines a canonical map $u_{A,B}:[A,B]\to[A_{\cc P},B_{\cc P}]$. 
Commutativity of diagrams~\eqref{B6.11} and~\eqref{fun} associated with maps $u_{A,B}$-s is directly verified. We see that
the localisation functor $(-)_{\cc P}:\Mod\cc A\to\Mod\cc A/\cc P$ is a $\cc V$-functor.

(3). By Lemma~\ref{bicomplete} $\Mod\cc A$ is a closed $\cc V$-module. The inclusion functor
$\Mod\cc A/\cc P\hookrightarrow\Mod\cc A$ is a $\cc V$-functor. It follows that the composite functor
   $$(\Mod\cc A/\cc P)\otimes\cc V\to(\Mod\cc A)\otimes\cc V\to\Mod\cc A\xrightarrow{(-)_{\cc P}}\Mod\cc A/\cc P$$
is a $\cc V$-functor (we tacitly use the previous statement here). The conditions of Definition~\ref{enri}
are plainly satisfied defining a right $\cc V$-module structure on $\Mod\cc A/\cc P$. It is also a closed $\cc V$-module with 
the coaction $\cc V$-functor defined by the composition
   $$\cc V^{\op}\otimes(\Mod\cc A/\cc P)\to\cc V^{\op}\otimes(\Mod\cc A)\to\Mod\cc A\xrightarrow{(-)_{\cc P}}\Mod\cc A/\cc P.$$
The adjointness conditions for the action and coaction functors of Definition~\ref{enri} are straightforward, as was to be shown.

(4). Suppose $M$ is $\cc P$-closed. Note that $[P,M]=0$ if and only if $\Hom_{\cc V}(G,[P,M])=0$,
where $G=\oplus_{g\in\cc G}g$ is the generator of $\cc V$. As $\cc P$ is enriched and $M$ is $\cc P$-torsionfree,
$\Hom_{\cc V}(G,[P,M])=\Hom_{\cc V}(P\oslash G,M)=0$. There is an exact sequence
   $$0\to M\to E_0\to E_1$$
with $E_0,E_1$ injective $\cc P$-torsionfree. One has a short exact sequence
   $$0\to M\xrightarrow{}E_0\to E_0/M\to 0$$
with $E_0/M$ being $\cc P$-torsionfree. As $\cc P$ is enriched, it follows that $[P,E_0/M]=0$ for all $P\in\cc P$.   
The long exact sequence
   $$0\to[P,M]\to[P,E_0]\to[P,E_0/M]\to\underline{\Ext}^1(P,M)\to\underline{\Ext}^1(P,E_0)=0$$
implies $[P,M]=\underline{\Ext}^1(P,M)=0$.

Conversely, suppose $[P,M]=\underline{\Ext}^1(P,M)=0$. Then $M$ is $\cc P$-torsionfree, and hence
the injective envelope $E(M)$ of $M$ is. As above, $[P,E(M)]=0$ for all $P\in\cc P$.
The long exact sequence
   $$0\to[P,M]\to[P,E(M)]\to[P,E(M)/M]\to\underline{\Ext}^1(P,M)=0$$
implies $[P,E(M)/M]=0$. Therefore $E(M)/M$ is $\cc P$-torsionfree, and hence its injective
envelope $E(E(M)/M)$ is. We see that $M$ is the kernel of a morphism between $\cc P$-closed objects,
hence it is $\cc P$-closed itself.
\end{proof}

Note that $\cc V$ can be regarded as the category of $\cc A$-modules with $\cc A=\{e\}$. In this case
enriched localising subcategories of $\cc V$ are also referred to as {\it tensor\/} in the literature.

\begin{cor}[Jerem\'ias L\'opez, L\'opez L\'opez and Villanueva N\'ovoa~\cite{JLV}]\label{monV}
For any tensor localising subcategory $\cc P$ of $\cc V$, the quotient $\cc V$-category $\cc V/\cc P$
is closed symmetric monoidal if we define the monoidal unit by $e_{\cc P}$, monoidal product by 
$V\boxtimes W:= (V\otimes W)_{\cc P}$
and the internal Hom-object by $\underline{\Hom}_{\cc V/\cc P}(V,W):=[V,W]$, $V,W\in\cc V/\cc P$.
In particular, the quotient $\cc V$-functor $\cc V\to\cc V/\cc P$ is strong symmetric monoidal.
Moreover, $\cc G_{\cc P}=\{g_{\cc P}\}_{g\in\cc G}$ is a set of dualizable generators of $\cc V/\cc P$.
\end{cor}

\begin{proof}
This follows from Theorem~\ref{svoistva}. In more detail, the associativity, left/right unit
isomorphisms for the monoidal structure on $\cc V/\cc P$ are, by definition, $\cc P$-localisations of the 
associativity, left/right unit isomorphisms for the monoidal structure on $\cc V$ (see~\cite[Section~6.1]{Bor}
for the corresponding definitions). We also use here the fact that the canonical morphism
   $$(\lambda_A\otimes B)_{\cc P}:(A\otimes B)_{\cc P}\to(A_{\cc P}\otimes B)_{\cc P}$$
is an isomorphism in $\cc V/\cc P$ (for this one uses Theorem~\ref{svoistva}(1)).
This also implies a natural isomorphism $[A,X]\cong[A_{\cc P},X]$ for any $A\in\cc V$, $X\in\cc V/\cc P$.

It remains to show that each $g_{\cc P}$ is dualizable. Consider the following exact sequence
   $$0\to P\to e\lra{\lambda_e}e_{\cc P}\to P'\to 0,$$
where $P,P'\in\cc P$. By Corollary~\ref{zabavnocor} the sequence
   $$0\to [g,P]\to g^\vee\xrightarrow{[g,\lambda_e]}[g,e_{\cc P}]\to [g,P']\to 0$$
is exact in $\cc V$ and $[g,P],[g,P']\in\cc P$. As $[g,e_{\cc P}]\cong[g_{\cc P},e_{\cc P}]$
is $\cc P$-closed by Theorem~\ref{svoistva}(1), we have that
$(g^\vee)_{\cc P}\cong (g_{\cc P})^\vee$ in $\cc V/\cc P$. Therefore one has for any $V\in\cc V/\cc P$
   $$(g_{\cc P})^\vee\boxtimes V\cong(g^\vee\otimes V)_{\cc P}\cong[g,V]_{\cc P}=[g,V]\cong[g_{\cc P},V].$$
We see that $g_{\cc P}$ is dualizable in $\cc V/\cc P$.
\end{proof}

\begin{ex}
Suppose $R$ is a commutative ring. Every localising subcategory $\cc P\subseteq\Mod R$
is tensor. Corollary~\ref{monV} implies that $\Mod R/\cc P$ is the category of right $\cc A$-modules
$[\cc A^{\op},\cc V]$, where $\cc A=\{R_{\cc P}\}$ and $\cc V=\Mod R/\cc P$. It is worth mentioning that,
in general, $\Mod R/\cc P$ is rarely equivalent to the category $\Mod R_{\cc P}=[R_{\cc P},\Ab]$ 
(see~\cite[Proposition~XI.3.4]{Stenstroem}), in which case we forget the
enriched information and regard $R_{\cc P}$ as an Ab-enriched ring and $\cc V=\Ab$.
\end{ex}

\section{The category of enriched generalised modules}\label{genmod}

The full subcategories of finitely presented right and left $\cc A$-modules will be denoted by
$\modd\cc A$ and $\cc A\modd$ respectively.

\begin{defs}\label{ac}
The {\it category of right (left) enriched generalised $\cc A$-modules\/} $\cc C_{\cc A}$ (${}_{\cc A}\cc C$)
is defined as the Grothendieck category of $\cc V$-functors $[\cc A\modd,\cc V]$ ($[\modd\cc A,\cc V]$),
where $\modd\cc A$ ($\cc A\modd$) is regarded as a $\cc V$-full subcategory of $\Mod\cc A$ ($\cc A\Mod$).
By definition, $\ac$ ($\cc C_{\cc A}$) is the category of left $\modd\cc A$-modules (left $\cc A\modd$-modules).
\end{defs}   
   
Note that if $\cc V=\Ab$ then $\cc C_{\cc A}$ is the category of generalised $\cc A$-modules
of~\cite[\S 6]{GG}. If, moreover, $\cc A$ has one object, in which case $\cc A$ is a ring $R$, $\cc C_{\cc A}$
coincides with the category of generalised modules $\cc C_R$ of~\cite{Herzog}. 

\begin{thm}\label{vazhno}
The category of generalised $\cc A$-modules $\cc C_{\cc A}$ is a locally coherent Grothendieck $\cc V$-category with
$\{[M,-]\mid M\in\cc A\modd\}$ being a family of coherent generators.
\end{thm}

\begin{proof}
By Theorem~\ref{pqp} $\cc C_{\cc A}$ is a Grothendieck $\cc V$-category with $\{[M,-]\oslash g\mid M\in\cc A\modd,g\in\cc G\}$
being a family of generators. As $\cc V$ is locally finitely presented and each object $g\in\cc G$ is finitely presented,
$[M,-]\oslash g$ is finitely presented as well due to isomorphisms
   $$\Hom_{\cc C_{\cc A}}([M,-]\oslash g,\lp_{i\in I}B_i)\cong\Hom_{\cc V}(g,\lp_{i\in I}B_i(M))=\lp_{i\in I}\Hom_{\cc V}(g,B_i(M))\cong\lp_{i\in I}([M,-]\oslash g, B_i).$$
Since $g$ is dualizable, it follows that
   $$[M,-]\oslash g\cong[M\oslash g^\vee,-].$$
We claim that $M\oslash g^\vee\in\cc A\modd$. Indeed, $M$ fits into an exact sequence
   $$\oplus_{i=1}^m[a_i,-]\oslash g_i\to\oplus_{i=1}^n[a_j,-]\oslash g_j\to M\to 0$$
for some $a_i,a_j\in\cc A$, $g_i,g_j\in\cc G$. As the sequence 
   $$\oplus_{i=1}^m[a_i,-]\oslash (g_i\otimes g^\vee)\to\oplus_{i=1}^n[a_j,-]\oslash(g_j\otimes g^\vee)\to M\oslash g^\vee\to 0$$
is exact and $\cc G$ is closed under tensor products and dual objects, our claim follows.

We see that $\cc C_{\cc A}$ is generated by $\{[M,-]\mid M\in\cc A\modd\}$. It remains to show that each generator $[M,-]$
is coherent. Let $A$ be a finitely generated subobject of $[M,-]$. There is $N\in\cc A\modd$ and an epimorphism 
$\eta:[N,-]\twoheadrightarrow A$. Composing it with the embedding $A\hookrightarrow[M,-]$, one gets a morphism
$\tau:[N,-]\to[M,-]$. Since 
   $$\Hom_{\cc C_{\cc A}}([N,-],[M,-])=\Hom_{\cc V}(e,[[N,-],[M,-]])\cong\Hom_{\cc V}(e,[M,N])=\Hom_{\cc A}(M,N),$$
it follows that $\tau=[f,-]$ for some $f\in\Hom_{\cc A}(M,N)$. Let $K=\coker f$, then $K\in\cc A\modd$ and $\kr\eta=[K,-]$
is finitely presented. We see that $A$ is finitely presented, and hence $[M,-]$ is coherent. This completes the proof.
\end{proof}

Let $M\in\Mod\cc A$ and $N\in\cc A\Mod$. By Theorem~\ref{polezno} one has
   $$M\cong\int^{\Ob\cc A}M(a)\oslash[-,a]\quad{\textrm{and}}\quad N\cong\int^{\Ob\cc A}\cc [b,-]\oslash N(b).$$
By definition,
   $$M\otimes_{\cc A}N:=\int^{a\in\Ob\cc A}\int^{b\in\Ob\cc A}M(a)\oslash[b,a]\oslash N(b)\in\Ob\cc V.$$
In particular, $[-,a]\otimes_{\cc A}[b,-]=[b,a]$, $M\otimes_{\cc A}[b,-]=M(b)$ and $[-,a]\otimes_{\cc A}N=N(a)$.
As coends in $\cc V$ are coequalizers, the latter three equalities uniquely determine $M\otimes_{\cc A}N$. In this
way we arrive at a bifunctor
   $$-\otimes_{\cc A}-:\Mod\cc A\times\cc A\Mod\to\cc V.$$
If $\gamma\in\Hom_{\cc A}(M,M')$, $\delta\in\Hom_{\cc A}(N,N')$, then
   $$\gamma\otimes\delta:=\int^{a\in\Ob\cc A}\int^{b\in\Ob\cc A}\gamma_a\oslash[b,a]\oslash\delta_b\in\Mor\cc V.$$
We use Corollary~\ref{good} here as well.

Observe that the restriction of $-\otimes_{\cc A}N$, $N\in\cc A\Mod$, to the 
full $\cc V$-subcategory $\modd\cc A$ of $\Mod\cc A$ is recovered as the enriched left Kan extension of the $\cc V$-functor
$N:\cc A\to\cc V$ along the full embedding $\cc A\hookrightarrow\modd\cc A$ due to the fact that $\modd\cc A$ is (skeletally) small
(see~\cite[Theorem~6.7.7]{Bor} and~\cite[Proposition~4.33]{Kelly}). 

We see that $-\otimes_{\cc A}N$ belongs to ${}_{\cc A}\cc C$. Though this fact will be enough for our purposes, 
nevertheless we want to show that the functor 
   $$-\otimes_{\cc A}N:\Mod\cc A\to\cc V$$
is actually a $\cc V$-functor between ``large" $\cc V$-categories. 

To this end, we have to construct morphisms in $\cc V$
   $$\alpha_{M,M'}:[M,M']\to[M\otimes_{\cc A}N,M'\otimes_{\cc A}N],\quad M,M'\in\Mod\cc A.$$
By using Yoneda's lemma~\cite[Theorem~1.3.3]{Bor1}, it is equivalent to constructing a natural transformation of functors
   \begin{equation}\label{nattr}
    \Hom_{\cc V}(-,[M,M'])\to\Hom_{\cc V}(-,[M\otimes_{\cc A}N,M'\otimes_{\cc A}N]).
   \end{equation}
For any $U\in\cc V$ and any $f:U\oslash M\to M'$, we get a morphism $f\otimes\id:U\oslash M\otimes_{\cc A}N\to M'\otimes_{\cc A}N$,
which uniquely corresponds to an element of $\Hom_{\cc V}(U,[M\otimes_{\cc A}N,M'\otimes_{\cc A}N])$.
As this correspondence is plainly natural in $U$, the desired natural transformation~\eqref{nattr} and morphisms $\alpha_{M,M'}$ are now constructed.

Next, for any $U,V\in\cc V$ one has bimorphisms
   $$b(U,V):(U,[M,M'])\otimes_{\bb Z}(V,[M',M''])\to(U\otimes V,[M,M''])$$
associated to the composition map $c_{M,M',M''}:[M,M']\otimes[M',M'']\to[M,M'']$. One also has bimorphisms
   $$b_N(U,V):(U,[M\otimes_{\cc A}N,M'\otimes_{\cc A}N])\otimes_{\bb Z}(V,[M'\otimes_{\cc A}N,M''\otimes_{\cc A}N])\to(U\otimes V,[M\otimes_{\cc A}N,M''\otimes_{\cc A}N])$$
associated to the composition map 
$c_{M\otimes N,M'\otimes N,M''\otimes N}:[M\otimes_{\cc A}N,M'\otimes_{\cc A}N]\otimes[M'\otimes_{\cc A}N,M'\otimes_{\cc A}N']\to[M\otimes_{\cc A}N,M''\otimes_{\cc A}N]$.
By construction, both bimorphisms guarantee commutativity of~\eqref{B6.11}. Commutativity of~\eqref{fun} is obvious.

Thus we get the following statement:

\begin{prop}\label{tensorN}
For any left $\cc A$-module $N$, the functor $-\otimes_{\cc A}N:\Mod\cc A\to\cc V$
is a $\cc V$-functor between  $\cc V$-categories. Moreover, the $\cc V$-functor 
   $$\cc A\Mod\to{}_{\cc A}\cc C,\quad N\mapsto-\otimes_{\cc A}N,$$
is $\cc V$-fully faithful.
\end{prop}

\begin{proof}
The first part has been verified above. We have also mentioned that the restriction of $-\otimes_{\cc A}N$, $N\in\cc A\Mod$, to the 
full $\cc V$-subcategory $\modd\cc A$ of $\Mod\cc A$ is recovered as the enriched left Kan extension of the $\cc V$-functor
$N:\cc A\to\cc V$ along the full embedding $\cc A\hookrightarrow\modd\cc A$, see~\cite[Theorem~6.7.7]{Bor}. As the enriched Kan extension 
is $\cc V$-fully faithful by the proof of ~\cite[Theorem~6.7.7]{Bor}, the second part follows as well.
\end{proof}

The following result relating $\cc A$-modules and generalised $\cc A$-modules is an enriched version for
~\cite[Proposition~7.1]{GG} and~\cite[Corollary~2.2]{GG1}.

\begin{thm}\label{recoll}
Define an enriched localizing subcategory 
$\cc S_{\cc A}:=\{Y\in\ac\mid Y(a)=0\textrm{ for all $a\in\cc A$}\}\subset\ac.$ There is a recollement 
\begin{diagram*}[column
sep=huge] \cc S_{\cc A}\arrow[r,"i"] &{\ac}
   \bendR{i_R}
   \bendL{i_L}
   \arrow[r,"r"]			
& {\cc A\Mod}\bendR{r_R} 
   \bendL{-\otimes_{\cc A}?} 
\end{diagram*}
with functors $i,r$ being the canonical inclusion and restriction functors
respectively. The functor $r_R$ is the enriched right Kan extension,
$i_R$ is the torsion functor associated with the localizing subcategory $\cc S_{\cc A}$.
Furthermore, if $\ac/\cc S_{\cc A}$ is the quotient category of $\ac$ with respect to $\cc S_{\cc A}$,
the functor $\cc A\Mod\to\ac/\cc S_{\cc A}$ sending $M$ to $(-\otimes_{\cc A}M)_{\cc S_{\cc A}}$
is an equivalence of categories.
\end{thm}

\begin{proof}
The theorem follows from~\cite[Theorems~3.4-3.5]{GJ2} and~\cite[Theorem~5.3]{AG} if we observe that
the left Kan extension functor $r_L:\cc A\Mod\to\ac$ equals the functor $M\mapsto-\otimes_{\cc A}M$.
\end{proof}

The Ziegler spectrum can be defined for arbitrary locally coherent Grothendieck categories (see~\cite{Herzog, Krause}).
Although $\cc C_{\cc A}$ is locally coherent by Theorem~\ref{vazhno}, and therefore the Ziegler spectrum 
$\Zg\cc C_{\cc A}$ in the sense of~\cite{Herzog, Krause} 
applies to $\ac$, this is actually not what we are going to investigate as $\Zg\cc C_{\cc A}$ does not capture the enriched
category information of $\cc C_{\cc A}$. Below we will define the Ziegler spectrum $\Zg_{\cc A}$ associated with $\cc C_{\cc A}$
that captures both the enriched category information of $\cc C_{\cc A}$ and the machinery of~\cite{Herzog, Krause}
(the points of $\Zg\cc C_{\cc A}$ and $\Zg_{\cc A}$ are the same but topologies are different).
To this end, we need to introduce and study pure-injective $\cc A$-modules and $\underline{\coh}$-injective objects.

\section{Pure-injective modules and $\underline{\coh}$-injective objects}\label{sectionpureinj}

\begin{defs}\label{cohinj}
A generalised $\cc A$-module $X\in\ac$ is said to be {\it $\underline{\coh}$-injective\/} if the functor 
$[-,X]:\coh\ac\to\cc V^{\op}$ is exact.
\end{defs}

\begin{thm}\label{cohinjobj}
The following conditions are equivalent for $X\in\ac$:
\begin{enumerate}
\item $X$ is $\underline{\coh}$-injective;

\item $X$ is right exact on $\modd\cc A$;

\item $X$ is isomorphic to $-\otimes_{\cc A}N$ for some $N\in\cc A\Mod$.
\end{enumerate}
\end{thm}

\begin{proof}
$(1)\Rightarrow(2)$. Given an exact sequence $K\xrightarrow{f} L\xrightarrow{g} M\to 0$ in $\modd\cc A$,
one has a long exact sequence in the Abelian $\cc V$-category $\coh\ac$
   \begin{equation}\label{ccoh}
    0\to[M,-]\xrightarrow{[f,-]}[L,-]\xrightarrow{[g,-]}[K,-]\xrightarrow{h}C\to 0.
   \end{equation}
As $X$ is $\underline{\coh}$-injective by assumption, the functor $[-,X]$ is exact on $\coh\ac$. Therefore 
one has a long exact sequence in $\cc V$
   $$0\to[C,X]\xrightarrow{h^*}X(K)\xrightarrow{[g,X]}X(L)\xrightarrow{[f,X]}X(M)\to 0.$$
We see that $X$ is right exact on $\modd\cc A$.

$(2)\Rightarrow(3)$. Suppose $X$ is right exact. Denote by $N$ the restriction of $X$ to the full $\cc V$-subcategory
$\cc A$ of $\modd\cc A$. Then $N\in\cc A\Mod$ and the counit of the adjunction gives a morphism
$\epsilon:-\otimes_{\cc A}N\to X$ in $\ac$. For any generator $g\oslash[-,a]$ of $\Mod\cc A$ one has
   \begin{multline*}
    X(g\oslash[-,a])=[[g\oslash[-,a],-],X]=[[g,[[-,a],-]],X]=[g^\vee\oslash[[-,a],-],X]=\\
    =[g^\vee,X([-,a])]=g\oslash X([-,a])=g\oslash N(a).
   \end{multline*}
Any $M\in\modd\cc A$ fits into an exact sequence
   $$\oplus_{i=1}^n g_i\oslash[-,a_i]\to\oplus_{j=1}^m g_j\oslash[-,a_j]\to M\to 0.$$
One gets a commutative diagramin $\cc V$
   $$\xymatrix{\oplus_{i=1}^n X(g_i\oslash[-,a_i])\ar[r]\ar@{=}[d]&\oplus_{j=1}^m X(g_j\oslash[-,a_j])\ar[r]\ar@{=}[d]&X(M)\ar[r] &0\\
                       \oplus_{i=1}^n g_i\oslash N(a_i)\ar[r]&\oplus_{j=1}^m g_j\oslash N(a_j)\ar[r]&M\otimes_{\cc A}N\ar[r]\ar[u]^\epsilon &0}$$
It follows that $\epsilon:-\otimes_{\cc A}N\to X$ is an isomorphism.

$(3)\Rightarrow(2)$. As the functor $\Mod\cc A\to\cc C_{\cc A}$, $M\mapsto M\otimes_{\cc A}-$, is left adjoint to the restriction functor,
it is right exact and preserves arbitrary colimits. If $N\in\cc A\modd$, it follows that 
$-\otimes_{\cc A}N$ is right exact. If $N$ is any left $\cc A$-module, then $N=\lp_I N_i$ with $N_i\in\cc A\modd$. As the direct limit
functor is exact in $\cc V$ and $-\otimes_{\cc A}N\cong\lp_I(-\otimes_{\cc A}N_i)$ is a 
direct limit of right exact functors, $-\otimes_{\cc A}N$ is right exact itself.

$(2)\Rightarrow(1)$. Let $C\in\coh\ac$ and let $X$ be right exact. Then it fits into an exact sequence of the form~\eqref{ccoh}. One has a diagram
   $$\xymatrix{0\ar[r]&[M,-]\ar[r]^{[f,-]}&[L,-]\ar[rr]^{[g,-]}\ar@{->>}[rd]_\ell&&[K,-]\ar[r]^{h}&C\ar[r]&0\\
                                                                   &&&\kr h\ar@{>->}[ur]_k}$$
It induces an exact sequence in $\cc V$
   $$0\to[C,X]\xrightarrow{h^*} X(K)\xrightarrow{k^*}[\kr h,X].$$ 
Since $X$ is right exact, the row of the diagram
      $$\xymatrix{X(K)\ar[dr]_{k^*}\ar[rr]^{X(g)}&&X(L)\ar[r]^{X(f)}&X(M)\ar[r]&0\\
                                                                   &[\kr h,X]\ar@{>->}[ur]_{\ell^*}}$$
is exact, and hence $\im X(g)=\coker(\textrm{ker}\,X(g))\cong\kr(\textrm{coker}\, X(g))=\kr X(f)=[\kr h,X]$. 
We see that $k^*$ is an epimorphism.

The functor between Abelian categories $[-,X]:\coh\ac\to\cc V^{\op}$ is right exact. The above arguments yield the following properties:

(1) the collection of objects $\cc R=\{[M,-]\mid M\in\modd\cc A\}$ is adapted to the functor $[-,X]$ in the sense of~\cite[Section~III.6.3]{GM},
and hence it induces a triangulated functor between derived categories $L[-,X]:D^-(\coh\ac)\to D^-(\cc V^{\op})=D^+(\cc V)$ --- see~\cite[Section~III.6.7]{GM};

(2) $[-,X]$ takes any resolution $R_\bullet\to C$ of $C\in\coh\ac$ by objects in $\cc R$ to an acyclic complex
$[C,X]\to[R_\bullet,X]$ in $\Ch^+(\cc V)$;

A short exact sequence $0\to A_1\to A_2\to A_3\to 0$ in $\coh\ac$ gives rise to a triangle 
$A_3[-1]\to A_1\to A_2\to A_3$ in $D^-(\coh\ac)$. Similarly to constructing projective resolutions there is a resolution
$R_\bullet^i: [M_i^2,-]\hookrightarrow[M_i^1,-]\to[M_i^0,-]$ of each $A_i$, $i=1,2,3$, such that the 
short exact sequence $0\to A_1\to A_2\to A_3\to 0$ fits to a commutative diagram of complexes
   $$\xymatrix{R_\bullet^1\ar[r]^\alpha\ar@{->>}[d]&R_\bullet^2\ar[r]^\beta\ar@{->>}[d]&R_\bullet^3\ar@{->>}[d]\\
                       A_1\ar@{ >->}[r]&A_2\ar@{->>}[r]&A_3}$$
The canonical functor $K^-(\cc R)[S_{\cc R}^{-1}]\to D^-(\coh\ac)$ is an equivalence of triangulated categories by~\cite[Proposition~III.6.4]{GM},
where $S_{\cc R}^{-1}$ is the class of quasi-isomorphisms in $K^-(\cc R)$. It follows that the sequence of resolutions above 
fits to a triangle 
   $$R_\bullet^3[-1]\to R_\bullet^1\xrightarrow{\alpha} R_\bullet^2\xrightarrow{\beta} R_\bullet^3$$ 
in $K^-(\cc R)[S_{\cc R}^{-1}]$. After applying the functor $L[-,X]$ to it, one gets a triangle in $D^+(\cc V)$
   $$[R_\bullet^3,X]\xrightarrow{\beta^*}[R_\bullet^2,X]\xrightarrow{\alpha^*}[R_\bullet^1,X]\to [R_\bullet^3,X][1].$$
This triangle induces a long exact sequence in $\cc V$ of cohomology objects  
   $$0\to H^0([R_\bullet^3,X])=[A_3,X]\to H^0([R_\bullet^2,X])=[A_2,X]\to H^0([R_\bullet^1,X])=[A_1,X]\to H^1([R_\bullet^3,X])=0.$$
Thus $[-,X]$ is an exact functor, as was to be shown. 
\end{proof}

\begin{defs}\label{puredef}
We say that an $\cc A$-homomorphism $f:K\to L$ in $\cc A\Mod$ is a {\it pure-monomorphism\/} if for any $M\in\modd\cc A$
the induced morphism $M\otimes f:M\otimes_{\cc A}K\to M\otimes_{\cc A}L$ in $\cc V$ is a monomorphism or, equivalently, 
$-\otimes f:-\otimes_{\cc A}K\to-\otimes_{\cc A}L$ is a monomorphism in $\ac$.
A left $\cc A$-module $Q$ is said to be {\it pure-injectiye\/} if every pure-monomorphism $p:Q\to N$ is a split monomorphism.
\end{defs}

\begin{cor}\label{pureinj}
An object $E\in\ac$ is injective if and only if it is isomorphic to one of the functors $-\otimes_{\cc A}Q$, 
where $Q$ is a pure-injective left $\cc A$-module.
\end{cor}

\begin{proof}
The proof literally repeats that of~\cite[Proposition~4.1]{Herzog}. In detail, if
$E\in\ac$ is injective, then it is $\underline{\coh}$-injective by Lemma~\ref{tochen}. Theorem~\ref{cohinjobj} implies it is isomorphic to one of the functors 
$-\otimes_{\cc A}Q$ with $Q\in\cc A\Mod$. The injective hypothesis implies that $Q$ must be pure-injective
(we implicitly use Proposition~\ref{tensorN} here).

Conversely, suppose $Q$ is pure-injective and $\alpha:-\otimes_{\cc A}Q\to X$ is a monomorphism in $\ac$.
Let $E$ be the injective envelope of $X$. By the first part of the proof $E\cong-\otimes_{\cc A}N$ for some 
pure-injective $N\in\cc A\Mod$. Then the composite monomorphism $-\otimes_{\cc A}Q\bl{\alpha}\hookrightarrow X\hookrightarrow-\otimes_{\cc A}N$
splits in $\ac$ as the pure-monomorphism $Q\hookrightarrow N$ splits in $\cc A\Mod$ (we implicitly use Proposition~\ref{tensorN} here).
It follows that $\alpha$ is a split monomorphism.
\end{proof}

We have now collected all the necessary information to pass to the definition of the Ziegler spectrum of $\cc A$.

\section{The Ziegler spectrum of an enriched ringoid}\label{sectionzg}

As $\ac$ is locally coherent, the full $\cc V$-subcategory of coherent objects $\coh\ac$ is Abelian.
A full $\cc V$-subcategory $\cc S$ of $\coh\ac$ is called an {\it enriched Serre subcategory\/} if it is
a Serre subcategory after forgetting the $\cc V$-structure and $m\oslash A\in\cc S$ for any finitely
presented object $m$ of $\cc V$ and any $A\in\cc S$. Recall from~\cite{Herzog, Krause} that any localizing 
subcategory of finite type in $\ac$ is of the form $\vec{\cc S}$, where $\cc S$ is a Serre subcategory and every object
of $\vec{\cc S}$ is a direct limit of objects from $\cc S$. 

\begin{prop}\label{serreV}
The following statements are equivalent for a localising subcategory of finite type $\vec{\cc S}$ of $\ac$:
\begin{enumerate}
\item $\vec{\cc S}$ is enriched;
\item ${\cc S}$ is an enriched Serre subcategory;
\item $g\oslash A\in\cc S$ for any generator $g\in\cc G$ and $A\in\cc S$.
\end{enumerate}
\end{prop}

\begin{proof}
$(1)\Rightarrow(2)$. The only thing to verify here is to show that $m\oslash A\in\coh\ac$ for any  
$m\in\fp(\cc V)$ and $A\in\cc S$ as $m\oslash A\in\vec{\cc S}$ by assumption.
Using Lemma~\ref{fprV}, one has
   $$(m\oslash A,\lp_I X_i)=(m,[A,\lp_IX_i])=\lp_I(m,[A,X_i])=\lp_I(m\oslash A,X_i).$$

$(2)\Rightarrow(3)$. This is straightforward.

$(3)\Rightarrow(1)$. Let $M\in\cc V$ and $X\in\vec{\cc S}$. Then $M=\lp_Im_i$ and $X=\lp_JX_j$
for some $m_i\in\fp(\cc V)$ and $X_j\in\cc S$. It follows that $M\oslash X\cong\lp_{I,J}m_i\oslash X_j$.
It is enough to check that each $m_i\oslash X_j$ is in $\cc S$.
Each $m_i$ fits in an exact sequence
   $$\oplus_{s=1}^mg_s\to\oplus_{t=1}^ng_t\to m_i\to 0,\quad g_s,g_t\in\cc G.$$
It induces an epimorphism $\oplus_{t=1}^n(g_t\oslash X_j)\twoheadrightarrow m_i\oslash X_j$. Since
each $g_t\oslash X_j\in\cc S$ by assumption and $\cc S$ is a Serre subcategory,
we see that $m_i\oslash X_j\in\cc S$, as required.
\end{proof}

By theorems of Herzog~\cite{Herzog} and Krause~\cite{Krause} there is a bijective correspondences between
Serre subcategories in $\coh\cc C$ and localising subcategories $\cc T$ of $\cc C$ of finite type, where $\cc C$
is a locally coherent Grothendieck category. Combining this correspondence with the preceding 
proposition, one gets the following statement.

\begin{cor}\label{serreVcor1}
There is an inclusion-preserving bijective correspondence between enriched
Serre subcategories $\cc S$ of $\coh\ac$ and enriched localising subcategories $\cc T$ of $\ac$ of finite type. This correspondence 
is given by the functions
   $$\cc S\mapsto\vec{\cc S},\quad\cc T\mapsto\cc T\cap\coh\ac,$$
which are mutual inverses.
\end{cor}

Denote by ${}_{\cc A}\Zg$ (respectively $\Zg_{\cc A}$) the 
set of the isomorphism classes of indecomposable pure-injective modules of $\cc A\Mod$
(respectively $\Mod{\cc A}$). To an arbitrary $\cc V$-subcategory $\cc X$ of $\coh\ac$, we associate the subset of ${}_{\cc A}\Zg$,
   $$\cc O(\cc X)=\{Q\in{}_{\cc A}\Zg\mid\textrm{for some $C$ in $\cc X$}, [C,-\otimes_{\cc A}Q]\ne 0\}.$$
If $\cc X=\{C\}$ is a singleton, we write $\cc O(C)$ to denote $\cc O(\cc X)$. Thus $\cc O(\cc X)=\bigcup_{C\in\cc X}\cc O(C)$.
Denote by $\surd\cc X$  (respectively $\langle\cc X\rangle$) the smallest Serre subcategory
(respectively the smallest enriched Serre subcategory) of $\coh\ac$ containing $\cc X$.

Recall that an object $A\in\coh\ac$ is a {\it subquotient\/} of $B\in\coh\ac$, if there is a filtration of $B$ by coherent subobjects
$B=B_0\geq B_1\geq B_2\geq 0$ such that $A\cong B_1/B_2$.

\begin{lem}\label{subq}
$\langle\cc X\rangle=\surd(\cc G\oslash\cc X)$, where $\cc G\oslash\cc X=\{g\oslash X\mid g\in\cc G, X\in\cc X\}$.
\end{lem}

\begin{proof}
Clearly, $\langle\cc X\rangle$ contains $\surd(\cc G\oslash\cc X)$.
By~\cite[Proposition~3.1]{Herzog} a coherent object $C\in\surd(\cc G\oslash\cc X)$ if and only if there are a finite filtration of $C$ by coherent subobjects
   $$C=C_0\supseteq C_1\supseteq\cdots\supseteq C_n= 0$$
and, for every $i<n$, $A_i\in\cc G\oslash\cc X$ such that $C_i/C_{i+1}$ is a subquotient of $A_i$. 
It follows from Corollary~\ref{zabavnocor} that the functor $g\oslash-:\coh\ac\to\coh\ac$ is exact. In particular, it respects filtrations as above. Therefore 
$\surd(\cc G\oslash\cc X)$ is an enriched Serre subcategory of $\coh\ac$ by Proposition~\ref{serreV}, and hence $\surd(\cc G\oslash\cc X)$ 
contains $\langle\cc X\rangle$.
\end{proof}

\begin{lem}\label{nuzhno}
If $A,B\in\coh\ac$ and $A$ is a subquotient of $B$, then 
$\cc O(A)\subset\cc O(B)$. If $0\to A\to B\to C\to 0$ is a short exact sequence in $\coh\ac$, 
then $\cc O(B)=\cc O(A)\cup\cc O(C)$.
\end{lem}

\begin{proof}
This immediately follows from Lemma~\ref{tochen} saying that $[-,E]:\coh\ac\to\cc V$
is an exact functor for every injective object of $\ac$.
\end{proof}

\begin{cor}\label{skobki}
For any $\cc V$-subcategory $\cc X\subseteq\coh\ac$, $\cc O(\cc X)=\cc O(\langle\cc X\rangle)$.
\end{cor}

\begin{proof}
It follows from Lemma~\ref{nuzhno} and the proof of Lemma~\ref{subq} that
$\cc O(\cc G\oslash\cc X)=\cc O(\langle\cc X\rangle)$. Suppose $[g\oslash C,E]\ne 0$ for some
$g\in\cc G$ and $C\in\cc X$. As $\cc G$ is a family of generators, there is $g'\in\cc G$ such that
$\Hom_{\cc V}(g',[g\oslash C,E])\ne 0$. Then $\Hom_{\cc V}(g'\otimes g,[C,E])\ne 0$, and hence $[C,E]\ne 0$,
because $g'\otimes g\in\cc G$. We see that
$E\in\cc O(\cc X)$.
\end{proof}

\begin{lem}\label{nenol}
If $C,E\in\ac$, then $[C,E]\ne0$ if and only if there is $g\in\cc G$ such that $\Hom_{\ac}(g\oslash C,E)\ne0$.
\end{lem}

\begin{proof}
Suppose $[C,E]\ne0$. Since $\cc G$ is a family of generators of $\cc V$, there is $g\in\cc G$ such that
$0\ne\Hom_{\ac}(g,[C,E])\cong\Hom_{\ac}(g\oslash C,E)$. Conversely, if $\Hom_{\ac}(g\oslash C,E)\ne0$
for some $g\in\cc G$ then $[C,E]\ne0$.
\end{proof}

If $\vec{\cc S}$ is a localising subcategory of finite type in $\ac$, denote by $t_{\cc S}:\ac\to\vec{\cc S}$
the torsion functor associated to $\vec{\cc S}$. Without loss of generality (for this use Lemma~\ref{nenol}) 
we can restrict our discussion to 
subsets of the form
   $$\cc O(\cc S)=\{Q\in{}_{\cc A}\Zg\mid t_{\cc S}(-\otimes_{\cc A}Q)\ne0\},$$
where $\cc S$ is an enriched Serre subcategory of $\coh\ac$.

The following result is the enriched version of the celebrated Ziegler Theorem~\cite[Theore~4.9]{Ziegler} that
associates to a ring $R$ a topological space ${}_R\Zg$, which we refer to as the {\it Ziegler spectrum of the ring $R$}, 
whose points are the isomorphism classes of 
the pure-injective indecomposable left $R$-modules. This theorem has been treated by 
Herzog~\cite[Section~4]{Herzog} in terms of the category of generalised modules ${}_R\cc C=(\modd R,\Ab)$.
We also refer the reader to books by Prest~\cite{Prest,Prest2}.

\begin{thm}[Ziegler]\label{zieglersp}
The following statements are true:
\begin{enumerate}
\item The collection of subsets of ${}_{\cc A}\Zg$,
   $$\{\cc O(\cc S)\mid\cc S\subseteq\coh\ac\textrm{ is an enriched Serre subcategory}\},$$
satisfies the axioms for the open sets of a topology on ${}_{\cc A}\Zg$. This topological space is called the 
\emph{Ziegler spectrum of the enriched ringoid $\cc A$}.

\item The collection of open subsets $\{\cc O(C)\mid C\in\coh\ac\}$
satisfies the axioms for a basis of open subsets of the Ziegler spectrum. 
Furthermore, $\cc O(C)=\emptyset$ if and only if $C=0$.

\item An open subset $\cc O$ of ${}_{\cc A}\Zg$ is quasi-compact if and only if it is one of the basic 
open subsets $\cc O(C)$ with $C\in\coh\ac$.
\end{enumerate}
\end{thm}

\begin{proof}
$(1)$. Clearly, $\cc O(0)=\emptyset$ and $\cc O(\coh\ac)={}_{\cc A}\Zg$. By Corollary~\ref{skobki} one has
$\bigcup_{i\in I}\cc O(\cc S_i)=\cc O(\bigcup_{i\in I}\cc S_i)=\cc O(\langle\bigcup_{i\in I}\cc S_i\rangle)$.
The fact that $\cc O(\cc S_1)\cap\cc O(\cc S_2)=\cc O(\cc S_1\cap\cc S_2)$ literally repeats the proof 
of~\cite[Theorem~3.4]{Herzog} if we observe that the intersection of two enriched Serre subcategories
is an enriched Serre subcategory.

(2). The first part of this statement follows from the fact that $\cc O(\cc X)=\bigcup_{C\in\cc X}\cc O(C)$
for any $\cc V$-subcategory $\cc X$ of $\coh\ac$. Suppose $[C,-\otimes_{\cc A}Q]=0$ for any $Q\in{}_{\cc A}\Zg$.
It follows that $\Hom_{\ac}(e,[C,-\otimes_{\cc A}Q])=\Hom_{\ac}(C,-\otimes_{\cc A}Q)=0$. Then $C=0$
by~\cite[Corollary~3.5]{Herzog}.

(3). The proof is like that of~\cite[Corollary~3.9]{Herzog}. Suppose $\cc O$ is quasi-compact. By the previous 
statement $\cc O=\bigcup_{i\in I}\cc O(C_i)$. Then there is a finite subset $J$ of $I$ such that
$\cc O=\bigcup_{i\in J}\cc O(C_i)=\cc O(\sqcup_{i\in J} C_i)$. Conversely, suppose $\cc O(C)=\bigcup_{i\in I}\cc O(C_i)=
\cc O(\{C_i\mid i\in I\})=\cc O(\langle\{C_i\mid i\in I\}\rangle)=\{Q\in{}_{\cc A}\Zg\mid t_{\langle\{C_i\mid i\in I\}\rangle}(-\otimes_{\cc A}Q)\ne0\}
=\{Q\in{}_{\cc A}\Zg\mid t_{\langle\{C\}\rangle}(-\otimes_{\cc A}Q)\ne0\}$. 

By~\cite[Theorem~3.8]{Herzog} $\langle\{C\}\rangle=\langle\{C_i\mid i\in I\}\rangle$.
By Lemma~\ref{subq} 
$\langle\{C_i\mid i\in I\}\rangle=\surd(\cc G\oslash\{C_i\mid i\in I\})=\surd(\cc G\oslash\{C\})$.
Since $\cc G$ is a family of generators for $\cc V$ and $e\in\fp(\cc V)$, there is an epimorphism
$\oplus_{i=1}^ng_i\twoheadrightarrow e$ in $\cc V$. It induces an epimorphism $\oplus_{i=1}^ng_i\oslash C\twoheadrightarrow C$
in $\ac$. Since $\oplus_{i=1}^ng_i\oslash C\in\surd(\cc G\oslash\{C\})$,
one has $C\in\surd(\cc G\oslash\{C\})$.

We see that $C\in\surd(\cc G\oslash\{C_i\mid i\in I\})$. By~\cite[Proposition~3.1]{Herzog}
there are a finite filtration of $C$ by coherent subobjects
   $$C=D_0\supseteq D_1\supseteq\cdots\supseteq D_n= 0$$
and, for every $\ell<n$, $g_\ell\in\cc G$, $C_\ell\in\{C_i\mid i\in I\}$, such that $D_\ell/D_{\ell+1}$ is a subquotient of $g_\ell\oslash C_\ell$.
Since only finitely many of the $C_i$ are needed, there is a finite subset $J$ of $I$ such that
$C\in\surd(\cc G\oslash\{C_i\mid i\in J\})=\langle\cc G\oslash\{C_i\mid i\in J\}\rangle$. Using Lemma~\ref{nuzhno}, it follows that
$\cc O(C)\subset\bigcup_{i\in J}\cc O(C_i)$, and hence $\cc O(C)=\bigcup_{i\in J}\cc O(C_i)$.
\end{proof}

\section{Enriched Auslander--Gruson--Jensen Duality}

In his thesis~\cite{Sor} Sorokin constructed various bifunctors between enriched categories. As an application,
he gets an enriched version of the  Auslander--Gruson--Jensen Duality. 
In this section we treat this duality in our context and refer the reader to~\cite{Sor} for more general context.

Namely, define a $\cc V$-functor 
   $$D:(\coh\ac)^{\op}\to\coh\cc C_{\cc A}$$
by the rule
   $$D(C)(N):=[C,-\otimes_{\cc A}N],\quad C\in\coh\ac,\quad N\in\cc A\modd.$$
We have to check that $D(C)\in\coh\cc C_{\cc A}$. First note that for any $M\in\modd\cc A$ enriched
Yoneda's lemma implies
   $$D([M,-])(N)=[[M,-],-\otimes_{\cc A}N]\cong M\otimes_{\cc A}N.$$
Therefore $D([M,-])\cong M\otimes_{\cc A}-$. Given $C\in\coh\ac$, there is an exact sequence
   $$[L,-]\to[M,-]\to C\to 0$$
in $\coh\ac$.
As $-\otimes_{\cc A}N$ is $\underline{\coh}$-injective by Theorem~\ref{cohinjobj}, $D$ is exact and takes the
exact sequence to an exact sequence in $\coh\cc C_{\cc A}$
   $$0\to D(C)\to M\otimes_{\cc A}-\to L\otimes_{\cc A}-.$$
We see that $D(C)\in\coh\cc C_{\cc A}$ as claimed.

By construction, $D(-\otimes_{\cc A}M)(N)=[-\otimes_{\cc A}M,-\otimes_{\cc A}N]$. it follows from
Proposition~\ref{tensorN} that $D(-\otimes_{\cc A}M)\cong[M,-]$. By symmetry, consider 
      $$D':(\coh\cc C_{\cc A})^{\op}\to\coh\ac$$
defined similarly to $D$. Then $D'D$ and $DD'$ are exact functors respecting functors of the form $[M,-]$
(up to isomorphism). Therefore $D$ and $D'$ are mutually inverse $\cc V$-equivalences between $\cc V$-categories.

Thus we have proven the following result (cf.~\cite[Theorem~5.1]{Herzog}).

\begin{thm}[Auslander~\cite{Au2}, Gruson and Jensen~\cite{GJ}]\label{agj}
The $\cc V$-functor $D:(\coh\ac)^{\op}\to\coh\cc C_{\cc A}$ defined above 
puts the $\cc V$-categories $\coh\ac$ and $\coh\cc C_{\cc A}$ in duality. Moreover, 
if $M\in\modd\cc A$ and $N\in\cc A\modd$, we have that
$D([M,-])\cong M\otimes_{\cc A}-$ and $D(-\otimes_{\cc A}N)\cong[N,-]$.
\end{thm}

An object $X$ of an Abelian $\cc V$-category $\cc C$ is said to be {\it internally projective\/}
(respectively {\it internally injective}) if the functor $[X,-]:\cc C\to\cc V$ (respectively $[-,X]:\cc C\to\cc V^{\op}$) is exact.
By construction, $[M,-]$ is internally projective for $\ac$.
As every $C\in\coh\ac$ is covered by an object of the form $[M,-]$, the Abelian $\cc V$-category
$\coh\ac$ has enough internally projective objects. The following statement follows from Theorems~\ref{agj} and~\ref{cohinjobj}.

\begin{cor}\label{injcoh}
Every internally injective object of $\coh\ac$ is isomorphic to one of the $\cc V$-functors 
$-\otimes_{\cc A}N$, where $N\in\cc A\modd$. The Abelian $\cc V$-category $\coh\ac$ has enough internally  injective objects, 
that is, for every $C\in\coh\ac$, there is a monomorphism $\iota:C\hookrightarrow-\otimes_{\cc A}N$ 
with $N\in\cc A\modd$.
\end{cor}

Given an enriched Serre subcategory $\cc S\subseteq\coh\ac$, the full $\cc V$-subcategory
$D\cc S=\{DC\mid C\in\cc S\}$ is plainly Serre. If $g\in\cc G$, $C\in\coh\ac$ and $N\in\cc A\modd$, then
   $$g\oslash D(C)(N)=g\oslash[C,-\otimes_{\cc A}N]\cong[g^\vee\oslash C,-\otimes_{\cc A}N]=D(g^\vee\oslash C)(N).$$
Since $g^\vee\oslash C\in\cc S$, it follows that $g\oslash D(C)\in D\cc S$, and hence $D\cc S$ is an enriched
Serre subcategory of $\coh\cc C_{\cc A}$ by Proposition~\ref{serreV}.

We arrive at an enriched version of a theorem of Herzog~\cite[Theorem~5.5]{Herzog} (the proof of the
second half of the theorem literally repeats arguments given in~\cite[p.~536]{Herzog}).

\begin{thm}[Herzog]\label{dualserre}
There is an inclusion-preserving bijective correspondence between the enriched 
Serre subcategories of $\coh\ac$ and those of $\coh\cc C_{\cc A}$ given by
   $$\cc S\mapsto D\cc S.$$ 
The induced map $\cc O(\cc S)\mapsto\cc O(D\cc S)$ is an isomorphism 
between the topologies, that is, the respective algebras of open sets, of the left and right 
Ziegler spectra of $\cc A$. Furthermore, the duality $D:(\coh\ac)^{\op}\to\coh\cc C_{\cc A}$ 
induces dualities between respective $\cc V$-subcategories 
$D:\cc S^{\op}\to D\cc S$ and $D:(\coh\ac/\vec{\cc S})^{\op}\to\coh\cc C_{\cc A}/\vec{D\cc S}$ as given 
by the following commutative diagram of Abelian $\cc V$-categories:
   $$\xymatrix{0\ar[r]&\cc S\ar[r]\ar[d]_D&\coh\ac\ar[r]\ar[d]_D&\coh(\ac/\vec{\cc S})\ar[r]\ar[d]^D&0\\
                       0\ar[r]&D\cc S\ar[r]&\coh\cc C_{\cc A}\ar[r]&\coh(\cc C_{\cc A}/\vec{D\cc S})\ar[r]&0}$$
\end{thm}

We have collected all the necessary facts about the Ziegler spectrum of an enriched ringoid
and (coherent) generalised $\cc A$-modules in oder to pass to applications.

\section{The injective spectrum of an enriched ringoid}\label{sectioninj}

Recall that a localizing subcategory $\cc S$ of a Grothendieck
category $\cc C$ is {\it of finite type\/} (respectively {\it of
strictly finite type}) if the functor $i:\cc C/\cc S\to\cc C$
preserves directed sums (respectively direct limits). If $\cc C$
is a locally finitely generated (respectively, locally finitely
presented) Grothen\-dieck category and $\cc S$ is of finite type
(respectively, of strictly finite type), then $\cc C/\cc S$ is a
locally finitely generated (respectively, locally finitely
presented) Grothendieck category and
   $$\fg(\cc C/\cc S)=\{C_{\cc S}\mid C\in\fg\cc C\}\quad\textrm{(respectively
     $\fp(\cc C/\cc S)=\{C_{\cc S}\mid C\in\fp\cc C\}$)}.$$
If $\cc C$ is a locally coherent Grothendieck category then $\cc
S$ is of finite type \ifff it is of strictly finite type (see,
e.g.,~\cite[Theorem~5.14]{GG}). In this case $\cc C/\cc S$ is locally
coherent.

Suppose $\cc C$ is a locally finitely presented Grothendieck
category. Denote by $\Sp\cc C$ the set of the isomorphism classes of indecomposable injective objects in $\cc C$.
The Ziegler topology defined for locally coherent categories in~\cite{Herzog, Krause}
was extended to $\Sp\cc C$ in~\cite[Theorem~11]{GG3}. In detail,
the collection of open subsets of $\Sp\cc C$ is given by
   $$\{O(\cc S)\mid\cc S\subset\cc C \textrm{ is a localizing subcategory of finite type}\},$$
where
   $$O(\cc S)=\{E\in\Sp\cc C\mid t_{\cc S}(E)\neq 0\}.$$
   
Following these results, denote by $\Sp\cc A$, where $\cc A$ is an enriched ringoid, 
the set of the isomorphism classes of indecomposable injective objects in $\cc A\Mod$.
Given an enriched localizing subcategory of finite type $\cc S\subseteq\cc A\Mod$, we set
   $$\cc O(\cc S)=\{E\in\Sp\cc A\mid t_{\cc S}(E)\neq 0\}.$$
Similarly to Lemma~\ref{nenol} and the fact that $\cc S$ is  closed under tensor products with objects from $\cc V$,
one gets that
   $$\cc O(\cc S)=\{E\in\Sp\cc A\mid [{\cc S},E]\neq 0\}.$$
   
\begin{thm}[after~\cite{GG3}]\label{injsp}
The collection of subsets of $\Sp\cc A$,
   $$\{\cc O(\cc S)\mid\cc S\subset\cc A\Mod \textrm{ is an enriched localizing subcategory of finite type}\},$$
satisfies the axioms for the open sets of a topology on the
injective spectrum $\Sp\cc A$. Moreover,
   \begin{equation}\label{111}
    \cc S\longmapsto\cc O(\cc S)
   \end{equation}
is an inclusion-preserving bijection between the enriched localizing
$\cc V$-subcategories $\cc S$ of finite type in $\cc A\Mod$ and the open subsets
of $\Sp\cc A$.
\end{thm}

\begin{proof}
Note that $\cc O(0)=\emptyset$ and $\cc O(\cc
A\Mod)=\Sp\cc A$. The proof of~\cite[Theorem~11]{GG3} shows that 
$\cc O(\cc S_1)\cap\cc O(\cc S_2)=\cc O(\cc S_1\cap\cc
S_2)$. This is due to the fact that $\cc S_1\cap\cc S_2$ is
enriched of finite type. The proof also shows that
$\surd\cup_{i\in I}\cc S_i$ is of finite type
with each $\cc S_i$ enriched of finite type. 

Suppose $E\in\cc A\Mod$ is injective. It follows 
from~\cite[Theorem~11]{GG3} that $t_{\surd\cup_{i\in I}\cc S_i}(E)=0$
if and only if $t_{\cc S_i}(E)=0$ for all $i\in I$. By Lemma~\ref{gE} $[g,E]$
is injective for all $g\in\cc G$. Since each $\cc S_i$ is enriched,
$[g,E]$ is $\cc S_i$-torsionfree if $E$ is $\cc S_i$-torsionfree.
So $g\oslash S\in\surd\cup_{i\in I}\cc S_i$ for all $g\in\cc G$
and $S\in\surd\cup_{i\in I}\cc S_i$. Since every $V\in\cc V$ is covered by
$\oplus_Ig_i$ with each $g_i\in\cc G$, we see that $V\oslash S$ also belongs to
$\surd\cup_{i\in I}\cc S_i$. We conclude that $\surd\cup_{i\in I}\cc S_i$ is enriched.
It follows from~\cite[Theorem~11]{GG3} that $\cup_{i\in I}\cc O(\cc S_i)=\cc O(\cup_{i\in
I}\cc S_i)=\cc O(\surd\cup_{i\in I}\cc S_i)$ and that the map~\eqref{111} is
bijective.
\end{proof}

The proof of the preceding theorem also implies the following statement.

\begin{cor}\label{priyatno}
A localizing subcategory $\cc S$ of $\cc A\Mod$ is enriched if and only if
for any $\cc S$-closed injective $\cc A$-module $E$ and any $g\in\cc G$
the injective $\cc A$-module $[g,E]$ is $\cc S$-closed.
\end{cor}

We say that the enriched ringoid $\cc A$ is {\it left coherent\/} if the category of left $\cc A$-modules
$\cc A\Mod$ is locally coherent or, equivalently, $\cc A\modd$ is an Abelian $\cc V$-category.

A continuous function $f:X\to Y$ between topological spaces
is called an {\it embedding\/} if in its image factorization
   $$f:X\to f(X)\hookrightarrow Y$$
with the image $f(X)\hookrightarrow Y$
 equipped with the subspace topology, we have that 
$X\to f(X)$ is a homeomorphism. This is called a {\it closed embedding\/} if the image 
$f(X)\subset Y$ is a closed subset.

\begin{thm}\label{spvszg}
Let $\cc A$ be left coherent. Then the enriched localising subcategory of Theorem~\ref{recoll}
$\cc S_{\cc A}:=\{Y\in\ac\mid Y(a)=0\textrm{ for all $a\in\cc A$}\}$
is of finite type in $\ac$ and
the functor $\cc A\Mod\to\ac$, $M\mapsto-\otimes_{\cc A}M$,
induces a closed embedding of topological spaces $\Sp\cc A\hookrightarrow{}_{\cc A}\Zg$.
\end{thm}

\begin{proof}
By Theorem~\ref{recoll} 
the functor $\cc A\Mod\to\ac/\cc S_{\cc A}$ sending $M$ to $(-\otimes_{\cc A}M)_{\cc S_{\cc A}}$
is an equivalence of categories. Denote by $\cc P:=\cc S_{\cc A}\cap\coh\ac$. Then $\cc P$
is an enriched Serre subcategory of $\coh\ac$. If we literally repeat the proof of the implication $(1)\Rightarrow(4)$ 
of~\cite[Theorem~7.4]{GG}, we get that $\cc S_{\cc A}=\vec{\cc P}$, and hence $\cc S_{\cc A}$ is of finite type.
Since $\ac$ is locally coherent, enriched localising subcategories of finite type and of strictly finite type coincide
by~\cite[Theorem~5.14]{GG}. Our theorem now follows from~\cite[Proposition~12]{GG3} and Theorems~\ref{zieglersp}
and~\ref{injsp}.
\end{proof}

\section{The Ziegler spectrum of a scheme}\label{sectionscheme}

Consider the case $\cc A=\{e\}$, where $e$ is the monoidal unit of $\cc V$. Then $\cc A\Mod$ or
$\Mod\cc A$ is identified with $\cc V$ and $\cc A\modd/\modd\cc A$ is identified with $\fp(\cc V)$.
In this case the category of generalised $\cc A$-modules coincides with the Grothendieck category of
enriched functors $[\fp(\cc V),\cc V]$.

Denote by $\Zg_{\cc V}$ the Ziegler spectrum of the enriched ringoid $\cc A=\{e\}$. It consists of the
isomorphism classes of pure-injective objects of $\cc V$ equipped with topology of Theorem~\ref{zieglersp}.
We also refer the reader to~\cite[Section~6]{HO} for further properties of pure-injective objects in $\cc V$.

\begin{lem}\label{compact}
$\Zg_{\cc V}$ is a quasi-compact topological space.
\end{lem}

\begin{proof}
This follows from Theorem~\ref{zieglersp}(3) if we observe that $\Zg_{\cc V}=\cc O(-\otimes e)$.
\end{proof}

In~\cite{Jason} Darbyshire showed that if $R$ is a commutative ring then $[\modd R,\Mod R]$ is isomorphic
to the category of generalised $R$-modules ${}_{R}\cc C=(\modd R,\Ab)$. Since $\modd R$ is a symmetric monoidal
$\Mod R$-category, a theorem of Day~\cite{Day} says that ${}_{R}\cc C=(\modd R,\Ab)\cong[\modd R,\Mod R]$
is equipped with a closed symmetric monoidal structure $(\underline{\Hom}_{{}_{ R}\cc C},\odot,-\otimes_RR)$, where
$-\otimes_RR\cong\Hom_R(R,-)$ is a monoidal unit for the Day monoidal product $\odot$. Darbyshire~\cite{Jason}
proved that the Auslander--Gruson--Jensen Duality 
   $$D:(\coh{}_R\cc C)^{\op}\to\coh\cc C_R$$ 
is isomorphic to the internal Hom-functor $\underline{\Hom}_{{}_{ R}\cc C}(?,-\otimes_RR)$.

Likewise, $\fp(\cc V)$ is a symmetric monoidal $\cc V$-category, and hence $[\fp(\cc V),\cc V]$
is equipped with  Day's closed symmetric monoidal structure $(\underline{\Hom}_{[\fp(\cc V),\cc V]},\odot,-\otimes e)$~\cite{Day}, where
$-\otimes e\cong[e,-]$ is a monoidal unit with respect to the Day monoidal product $\odot$ --- see Theorem~\ref{daythm}.

\begin{thm}\label{dualityV}
The Auslander--Gruson--Jensen Duality of Theorem~\ref{agj}
   $$D:(\coh[\fp(\cc V),\cc V])^{\op}\to\coh[\fp(\cc V),\cc V]$$
is isomorphic to the internal Hom-functor $\underline{\Hom}_{[\fp(\cc V),\cc V]}(?,-\otimes e)$.
\end{thm}

\begin{proof}
There is an isomorphism of $\cc V$-functors 
   $$\underline{\Hom}_{[\fp(\cc V),\cc V]}(?,-\otimes e)\cong\underline{\Hom}_{[\fp(\cc V),\cc V]}(?,[e,-]).$$
For any $C\in\coh[\fp(\cc V),\cc V]$ and $N\in\fp(\cc V)$ one has
   $$\underline{\Hom}_{[\fp(\cc V),\cc V]}(C,-\otimes e)(N)=[C,[e,-](N\otimes-)]=[C,[e,N\otimes-]]\cong[C,N\otimes-]=D(C)(N).$$
The isomorphism induces a $\cc V$-natural transformation of $\cc V$-functors $\underline{\Hom}_{[\fp(\cc V),\cc V]}(?,-\otimes e)\to D$,
which is objectwise an isomorphism. Therefore it is an equivalence of $\cc V$-functors.
\end{proof}

\begin{lem}
The functor $\cc V\to[\fp(\cc V),\cc V]$, $N\mapsto-\otimes N$, is strong monoidal. In particular,
$(-\otimes M)\odot(-\otimes N)\cong-\otimes(M\otimes N)$.
\end{lem}

\begin{proof}
By Theorems~\ref{polezno}, \ref{daythm} 
one has $\cc V$-natural isomorphisms
   \begin{multline*}
      (-\otimes M)\odot(-\otimes N)\cong(\int^A[A,-]\otimes A\otimes M)\odot(\int^B[B,-]\otimes B\otimes N)\cong\\
      (\int^A[A,-]\otimes A\odot\int^B[B,-]\otimes B)\otimes(M\otimes N)\cong(-\otimes e)\odot(-\otimes e)\otimes(M\otimes N)\cong-\otimes(M\otimes N).
    \end{multline*}
Our lemma now follows.
\end{proof}

\begin{rem}\label{notdual}
It is worth mentioning that generators $[M,-]$, $M\in\fp(\cc V)$, of $[\fp(\cc V),\cc V]$ are not dualizable
with respect to the tensor product $\odot$ in general. Indeed, following~\cite[Example~7.5]{AG} consider 
$\cc V=\Ab$. Then $(\bb Z_2,-)\odot-\otimes\bb Z_2\cong(\bb Z_2,-)$ but 
$\underline{\Hom}_{\cc C_{\bb Z}}(-\otimes\bb Z_2,-\otimes\bb Z_2)\cong-\otimes\bb Z_2$.
\end{rem}

We are now in a position to pass to the construction of the Ziegler spectrum of a ``nice" scheme $X$.
We start with the closed symmetric monoidal category $(\Qcoh(X),\otimes_X,{\cc Hom}_X^{qc}$) of 
quasi-coherent scheaves over a scheme $X$. It is a Grothendieck category 
by~\cite[Lemma~1.3]{Alonso} (see~\cite[Corollary 3.5]{EE} as well).

\begin{defs}
A quasi-compact quasi-separated scheme $X$ satisfying the strong resolution property is called a {\it nice scheme}.
By Example~\ref{examples}(2) if $X$ is nice, $\Qcoh(X)$ has a family of 
dualizable generators $\cc G = \{g_i\}_{i\in I}$ such that the monoidal unit $\cc O_X$ is finitely presented. 
\end{defs}

Given a nice scheme $X$, pure-injective objects in $\Qcoh(X)$ in the sense of Definition~\ref{puredef} 
(the enriched ringoid $\cc A$ is the singleton $\{\cc O_X\}$ in this case) are the same 
with pure-injective quasi-coherent sheaves in the sense of~\cite[Definition~4.1]{EEO}.

\begin{defs}\label{zieglersch}
The {\it category of generalised quasi-coherent sheaves\/} $\cc C_X$ of a nice scheme $X$ 
is defined as $[\fp(\Qcoh(X)),\Qcoh(X)]$.

The {\it Ziegler spectrum of a nice scheme $X$}, denoted by $\Zg_X$, is the Ziegler spectrum $\Zg_{\cc V}$
associated with $\cc V=\Qcoh(X)$. By definition, the points of $\Zg_X$ are the isomorphism classes of indecomposable
pure-injective quasi-coherent sheaves. $\Zg_X$ is a topological space equipped with the Ziegler topology 
of Theorem~\ref{zieglersp}.

The {\it injective spectrum of $X$}, denoted by $\Sp(X)$, is the injective spectrum associated with $\Qcoh(X)$ equipped with topology
of Theorem~\ref{injsp} (the enriched ringoid $\cc A$ is the singleton $\{\cc O_X\}$ in this case). Observe that
the topological space $\Sp(X)$ is nothing but the topological space $\Sp_{\mathrm{fl},\otimes}(X)$ in the sense 
of~\cite[Theorem~19]{GG3} (in this paper we drop these subscripts for the injective spectrum of $X$ to simplify our notation).

Following~\cite{GG3} a scheme $X$ is {\it locally coherent\/} if it can be covered by
open affine subsets $\spec R_i$, where each $R_i$ is a coherent
ring. $X$ is {\it coherent\/} if it is locally coherent,
quasi-compact and quasi-separated. 
\end{defs}

The injective spectrum $\Sp(X)$ plays a prominent role for classifying 
finite localizations of quasicoherent sheaves and for the theorem reconstructing $X$ out of $\Qcoh(X)$ in~\cite{GG3}
(see~\cite{GaPr1,GaPr2,GaPr3} as well).
The following theorem relates quasi-coherent sheaves and generalised quasi-coherent sheaves.
It also relates injective and Ziegler spectra of $X$.

\begin{thm}\label{ugu}
The following statements are true for a nice scheme $X$:

(1) Define an enriched localizing subcategory 
$\cc S_{X}:=\{Y\in\cc C_X\mid Y(\cc O_X)=0\}\subset\cc C_X$. 
Then there is a recollement \begin{diagram*}[column
sep=huge] \cc S_{X}\arrow[r,"i"] &{\cc C_X}
   \bendR{i_R}
   \bendL{i_L}
   \arrow[r,"r"]			
& {\Qcoh(X)}\bendR{r_R} 
   \bendL{-\otimes_{X}?} 
\end{diagram*}
with functors $i,r$ being the canonical inclusion and restriction functors
respectively. The functor $r_R$ is the enriched right Kan extension,
$i_R$ is the torsion functor associated with the localizing subcategory $\cc S_{X}$.
Furthermore, if $\cc C_X/\cc S_{X}$ is the quotient category of $\cc C_X$ with respect to $\cc S_{X}$,
the functor $\Qcoh(X)\to\cc C_X/\cc S_{X}$ sending $M$ to $(-\otimes_{X}M)_{\cc S_{X}}$
is an equivalence of categories.

(2) The Ziegler spectrum $\Zg_X$ is a quasi-compact topological space and the Auslander--Gruson--Jensen
Duality of~Theorem~\ref{dualityV}
   $$D\cong\underline{\Hom}_{\cc C_X}(?,-\otimes_X\cc O_X):(\coh\cc C_X)^{\op}\to\coh\cc C_X$$
induces an isomorphism $\cc O(\cc S)\mapsto\cc O(D\cc S)$ between open sets of $\Zg_X$.

(3) If $X$ is coherent, $\cc S_X$ is of finite type and there is a closed embedding $\Sp(X)\hookrightarrow\Zg_X$ 
of the injective spectrum of $X$ into its Ziegler spectrum.
\end{thm}

\begin{proof}
This follows from Theorems~\ref{recoll}, \ref{dualserre}, \ref{spvszg}, Lemma~\ref{compact} and~\cite[Proposition~40]{GG3}.
\end{proof}

\begin{rem}
In~\cite{PS} Prest and Sl\'avik study the categorical Ziegler spectrum $\Zg(\Qcoh(X))$.
It is not quasi-compact~\cite[Corollary~5.7]{PS} in contrast with $\Zg_X$, but
the subset $\cc D_X$ of the indecomposable geometrically pure-injective quasicoherent sheaves form a closed 
quasi-compact subset of $\Zg(\Qcoh(X))$ \cite[Theorem~4.8]{PS}. It would be interesting
to compare $\Zg_X$ of Theorem~\ref{ugu} with $\cc D_X$ endowed with the subspace topology
of $\Zg(\Qcoh(X))$.
\end{rem}

Recall from~\cite{Hoc} that a topological space is {\it spectral\/}
if it is $T_0$, quasi-compact, if the quasi-compact open subsets are
closed under finite intersections and form an open basis, and if
every non-empty irreducible closed subset has a generic point. Given
a spectral topological space, $X$, Hochster~\cite{Hoc} endows the
underlying set with a new, ``dual", topology, denoted $X^*$, by
taking as open sets those of the form $Y=\bigcup_{i\in\Omega}Y_i$
where $Y_i$ has quasi-compact open complement $X\setminus Y_i$ for
all $i\in\Omega$. Then $X^*$ is spectral and $(X^*)^*=X$ (see
\cite[Proposition~8]{Hoc}). As an example, the underlying topological space (denote it by the same symbol) of a quasi-compact,
quasi-separated scheme $X$ is spectral.

\begin{cor}\label{noeth}
Let $X$ be a noetherian nice scheme. Then there is a natural closed embedding 
of topological spaces $X^*\hookrightarrow\Zg_X$.
\end{cor}

\begin{proof}
By a theorem of Gabriel~\cite[Chapter~IV]{Gab} the map $\Sp(X)\to X^*$ sending $E\in\Sp(X)$
to the generic point of the irreducible subset $\supp_X(E)$ of $X$ is a homeomorphism. Our statement now
follows from Theorem~\ref{ugu}(3).
\end{proof}

We conclude the paper with mentioning yet another topological space $\Zg_{\cc V,\odot}$ whose points
are those of $\Zg_{\cc V}$ but with open sets being in bijective correspondence with tensor-closed Serre
subcategories of $\coh[\fp(\cc V),\cc V]$ with respect to the monoidal product $\odot$ on $[\fp(\cc V),\cc V]$.
This topology is coarser than the Ziegler topology on $\Zg_{\cc V}$ as every tensor Serre subcategory is enriched. It 
literally repeats the construction of the tensor fl-topology on $\Sp(X)$ 
in the sense of~\cite[Theorem~19]{GG3}. 
In~\cite[\S~4.1]{Wag} Wagstaffe studied the Ziegler topology for closed symmetric monoidal categories
of additive functors $(\fp(\cc C),\Ab)$, where $\cc C$ is a finitely accessible tensor category. 
It is described in terms of tensor-closed Serre subcategories with respect to
Day's tensor product $\odot$.
The space $\Zg_{\cc V,\odot}$ should share lots of common properties with~\cite[\S~4.1]{Wag} and
we invite the interested reader to study the topological space $\Zg_{\cc V,\odot}$.
   

\end{document}